\documentclass[fleqn,reqno,11pt,a4paper,final]{amsart}

\usepackage[a4paper,left=20mm,right=20mm,top=20mm,bottom=20mm,marginpar=20mm]{geometry}
\usepackage{amsmath}
\usepackage{amssymb}
\usepackage{amsthm}
\usepackage{amscd}
\usepackage[utf8]{inputenc}
\usepackage{bbm}
\usepackage{color}
\usepackage[english=american]{csquotes}
\usepackage[final]{graphicx}
\usepackage{hyperref}
\usepackage{calc}
\usepackage{mathptmx}
\usepackage{mathtools}
\usepackage{tikz}
\usepackage{graphicx}
\usepackage{stix}

\usepackage[backend=biber,bibencoding=utf8,firstinits,maxnames=4,style=alphabetic,isbn=false, doi=false, eprint= true, url= false,date=year]{biblatex}

\bibliography{stability.bib}

\definecolor{luh-dark-blue}{rgb}{0.0, 0.313, 0.608}

\graphicspath{{../Pictures/}}

\numberwithin{equation}{section}

\newtheoremstyle{thmlemcorr}{10pt}{10pt}{\itshape}{}{\bfseries}{.}{10pt}{{\thmname{#1}\thmnumber{ #2}\thmnote{ (#3)}}}
\newtheoremstyle{thmlemcorr*}{10pt}{10pt}{\itshape}{}{\bfseries}{.}\newline{{\thmname{#1}\thmnumber{ #2}\thmnote{ (#3)}}}
\newtheoremstyle{remexample}{10pt}{10pt}{}{}{\bfseries}{.}{10pt}{{\thmname{#1}\thmnumber{ #2}\thmnote{ (#3)}}}
\newtheoremstyle{ass}{10pt}{10pt}{}{}{\bfseries}{.}{10pt}{{\thmname{#1}\thmnumber{ A#2}\thmnote{ (#3)}}}

\theoremstyle{thmlemcorr}
\newtheorem{theorem}{Theorem}
\numberwithin{theorem}{section}
\newtheorem{lemma}[theorem]{Lemma}

\newtheorem{proposition}[theorem]{Proposition}

\newtheorem{definition}[theorem]{Definition}

\theoremstyle{thmlemcorr*}
\newtheorem*{theorem*}{Theorem}
\newtheorem{lemma*}[theorem]{Lemma}
\newtheorem{corollary*}[theorem]{Corollary}
\newtheorem{proposition*}[theorem]{Proposition}
\newtheorem{problem*}[theorem]{Problem}
\newtheorem{conjecture*}[theorem]{Conjecture}
\newtheorem{definition*}[theorem]{Definition}
\newtheorem{assumption*}[theorem]{Assumption}

\theoremstyle{remexample}
\newtheorem{remark}[theorem]{Remark}

\theoremstyle{ass}

\newcommand{\Acal}{\mathcal{A}}

\newcommand{\Dcal}{\mathcal{D}}

\newcommand{\Fcal}{\mathcal{F}}

\newcommand{\Lcal}{\mathcal{L}}

\newcommand{\N}{\mathbb{N}}
\newcommand{\R}{\mathbb{R}}

\newcommand{\eps}{\varepsilon}

\def\XXint#1#2#3{{\setbox0=\hbox{$#1{#2#3}{\int}$}
\vcenter{\hbox{$#2#3$}}\kern-.5\wd0}}

\renewcommand{\eps}{\varepsilon}
\renewcommand{\phi}{\varphi}
\begin{document}

%% TITLE MATTERS

\title[]{Long-time behaviour and stability for quasilinear doubly degenerate parabolic equations of higher order}

\author{Jonas Jansen}
\address{\textit{Jonas Jansen:}  Institute of Applied Mathematics, University of Bonn, Endenicher Allee~60, 53115 Bonn, Germany}
\email{jansen@iam.uni-bonn.de}

\author{Christina Lienstromberg}
\address{\textit{Christina Lienstromberg:}  Institute of Analysis, Dynamics and Modeling, University of Stuttgart, Pfaffenwaldring~57, 70569 Stuttgart, Germany}
\email{christina.lienstromberg@iadm.uni-stuttgart.de}

\author{Katerina Nik}
\address{\textit{Katerina Nik:} Faculty of Mathematics, University of Vienna, Oskar-Morgenstern-Platz 1, 1090 Vienna, Austria}
\email{katerina.nik@univie.ac.at}

\begin{abstract}
We study the long-time behaviour of solutions to quasilinear doubly degenerate parabolic problems of fourth order. The equations model for instance the dynamic behaviour of a non-Newtonian thin-film flow on a flat impermeable bottom and with zero contact angle. We consider a shear-rate dependent fluid the rheology of which is described by a constitutive power-law or Ellis-law for the fluid viscosity. In all three cases, positive constants (i.e. positive flat films) are the only positive steady-state solutions. Moreover, we can give a detailed picture of the long-time behaviour of solutions with respect to the $H^1(\Omega)$-norm.
In the case of shear-thickening power-law fluids, one observes that solutions which are initially close to a steady state, converge to equilibrium in finite time. In the shear-thinning power-law case, we find that steady states are polynomially stable in the sense that, as time tends to infinity, solutions which are initially close to a steady state, converge to equilibrium at rate $1/t^{1/\beta}$ for some $\beta > 0$. Finally, in the case of an Ellis-fluid, steady states are exponentially stable in $H^1(\Omega)$.
\end{abstract}
\vspace{4pt}

\maketitle



%\vspace{4pt}

%\noindent\textsc{Date:} \today{}.
%\end{abstract}

%% PDF MATTERS

%% START OF CONTENT

%=============================================================================
%=============================================================================
%=============================================================================

\section{Introduction}

\subsection{Aim of the paper} 
The present paper is concerned with the asymptotic behaviour of positive weak solutions to fourth-order quasilinear (doubly) degenerate parabolic problems as they arise in the modelling of non-Newtonian thin-film flows. It turns out that, for large times, fluids with a shear-rate dependent viscosity exhibit a specific asymptotic behaviour, depending on their shear-thickening or shear-thinning nature, respectively.

We consider a thin layer of a viscous, non-Newtonian and incompressible fluid on an impermeable flat bottom, as sketched in Figure \ref{fig:thin-film}.

%----------------------------------------------------
%----------------------------------------------------

\begin{center}
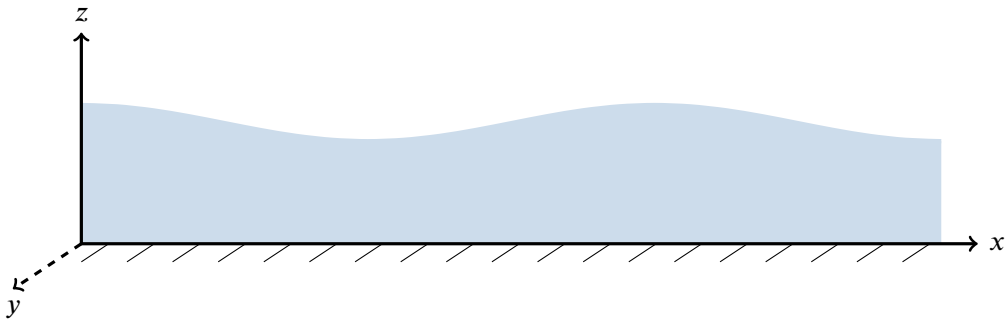
\begin{figure}[h]
\begin{tikzpicture}[domain=0:3*pi, xscale=1.2, yscale=0.8] 
\draw[ultra thick, smooth, variable=\x, luh-dark-blue!20] plot (\x,{0.3*cos(\x r)+2}); 
\fill[luh-dark-blue!20] plot[domain=0:3*pi] (\x,{0}) -- plot[domain=3*pi:0] (\x,{0.3*cos(\x r)+2});
\draw[very thick,<->] (3*pi+0.4,0) node[right] {$x$} -- (0,0) -- (0,3.5) node[above] {$z$};
\draw[very thick,dashed,->] (0,0) -- (-0.75,-0.75) node[below] {$y$};
\draw[-] (0,-0.3) -- (0.3, 0);
\draw[-] (0.5,-0.3) -- +(0.3, 0.3);
\draw[-] (1,-0.3) -- +(0.3, 0.3);
\draw[-] (1.5,-0.3) -- +(0.3, 0.3);
\draw[-] (2,-0.3) -- +(0.3, 0.3); 
\draw[-] (2.5,-0.3) -- +(0.3, 0.3);
\draw[-] (3,-0.3) -- +(0.3, 0.3);
\draw[-] (3.5,-0.3) -- +(0.3, 0.3);
\draw[-] (4,-0.3) -- +(0.3, 0.3);
\draw[-] (4.5,-0.3) -- +(0.3, 0.3);
\draw[-] (5,-0.3) -- +(0.3, 0.3);
\draw[-] (5.5,-0.3) -- +(0.3, 0.3);
\draw[-] (6,-0.3) -- +(0.3, 0.3);
\draw[-] (6.5,-0.3) -- +(0.3, 0.3);
\draw[-] (7,-0.3) -- +(0.3, 0.3);
\draw[-] (7.5,-0.3) -- +(0.3, 0.3);
\draw[-] (8,-0.3) -- +(0.3, 0.3);
\draw[-] (8.5,-0.3) -- +(0.3, 0.3);
\draw[-] (9,-0.3) -- +(0.3, 0.3);
\end{tikzpicture}   
\caption{Cross section of fluid film on impermeable solid bottom.}
\label{fig:thin-film}
\end{figure} 
\end{center}

%----------------------------------------------------
%----------------------------------------------------

In addition to the non-Newtonian fluid rheology, the following modelling assumptions are crucial for the analysis of the resulting partial differential equations. First, the fluid flow is assumed to be uniform in one horizontal direction (in $y$-direction in Figure \ref{fig:thin-film}), such that we obtain a (spatially) one-dimensional problem. Moreover, we assume that the characteristic height of the fluid layer is rather thin compared to its characteristic length and consider the asymptotic limit of a vanishing aspect ratio. Based on a non-Newtonian Navier--Stokes system, we use the so-called lubrication approximation \cite{giacomelli_rigorous_2003,gunther_justification_2008,ockendon_viscous_1995} in order to derive an evolution equation for the height $u = u(t,x)\geq 0$ of the fluid film at time $t > 0$ and spatial position $x \in \Omega$, where $\Omega \subset \R$ is a bounded interval. We neglect gravitational effects and assume that the dynamics of the flow is driven by capillarity only. Finally, we prescribe a no-slip condition on the lower boundary of the fluid film. However, the mathematical analysis of the present paper does also apply to the case of Navier-slip conditions. 

As constitutive laws for the non-Newtonian shear-dependent fluid we consider so-called \textbf{power-law fluids}, also called \textbf{Ostwald-de Waele fluids}, and so-called \textbf{Ellis-fluids}; see below for more details on these material laws. In the case of power-law fluids, when prescribing a no-slip condition on the lower boundary, the resulting evolution problem reads
\begin{equation} \label{eq:PDE_power-law}
	\begin{cases}
		u_t + \bigl(u^{\alpha+2} |u_{xxx}|^{\alpha-1} u_{xxx}\bigr)_x = 0, &
		t > 0,\ x \in \Omega,
		\\
		u_x(t,x) = u_{xxx}(t,x) = 0,
		&
		t > 0,\ x \in \partial\Omega,
		\\
		u(0,x) = u_0(x), 
		&
		x \in \Omega.
	\end{cases}
\end{equation}
Note that $\eqref{eq:PDE_power-law}_1$ is a fourth-order quasilinear parabolic equation that is doubly degenerate in the sense that the degeneracy occurs both with respect to the unknown $u$ and with respect to its third spatial derivative $u_{xxx}$. 
The Neumann-type boundary conditions $u_x = u_{xxx} = 0$ on $\partial\Omega$ reflect the zero-contact angle condition and the no-flux condition at the lateral boundary, respectively. Finally, $u_0 > 0$ denotes the given positive initial film height. Note that for $0 < \alpha<1$, the coefficients of the highest-order term depend only H\"older continuously on the unknown and lower-order derivatives.

In the case of Ellis-fluids, we obtain the evolution equation
\begin{equation} \label{eq:PDE_Ellis}
	\begin{cases}
		u_t + a\bigl(u^3 \bigl[1+b|uu_{xxx}|^{\alpha-1}\bigr] u_{xxx}\bigr)_x = 0, &
		t > 0,\ x \in \Omega,
		\\
		u_x(t,x) = u_{xxx}(t,x) = 0,
		&
		t > 0,\ x \in \partial\Omega,
		\\
		u(0,x) = u_0(x), 
		&
		x \in \Omega.
	\end{cases}
\end{equation}
Here, $a,b > 0$ are positive physical parameters, depending on the constant surface tension, the flow-behaviour exponent $\alpha$ and the characteristic viscosity of the fluid. However, for clarity of presentation, we drop these parameters in our analysis since they do not affect our arguments. This equation has for instance been studied in \cite{ansini_shear-thinning_2002,lienstromberg_local_2020} in the context of self-similar solutions and local strong solutions, respectively.

The main difference in the classification of \eqref{eq:PDE_power-law} and \eqref{eq:PDE_Ellis} is that \eqref{eq:PDE_power-law} is doubly degenerate in the sense that we loose parabolicity if either the unknown $u$ or its third spatial derivative $u_{xxx}$ become zero. In contrast, \eqref{eq:PDE_Ellis} is degenerate only in the unknown $u$ itself. 

For $\alpha=1$ we recover in both equations \eqref{eq:PDE_power-law} and \eqref{eq:PDE_Ellis} the well-known Newtonian thin-film equation
\begin{equation} \label{eq:PDE_Newtonian}
	u_t + \bigl(u^3 u_{xxx}\bigr)_x = 0, 
	\quad
	t > 0,\ x \in \Omega.
\end{equation}
This equation is studied extensively in the mathematical literature. For results concerning existence, uniqueness and stability of weak solutions to \eqref{eq:PDE_Newtonian} we refer the reader for instance to the works \cite{bernis_higher_1990,beretta_nonnegative_1995,bertozzi_lubrication_1996}.

%-------------------------------------
%-------------------------------------
\bigskip

\subsection{Main results of the paper -- Stability of steady states and long-time behaviour of positive weak solutions} 
In the present paper we study the behaviour of positive weak solutions to \eqref{eq:PDE_power-law} and \eqref{eq:PDE_Ellis}, respectively, for large times. Note that we consider only the case of strictly positive initial values $u_0 > 0$ since these allow us to find a positive time up to which solutions remain strictly positive. 

The main results of the paper are the following: We prove local existence of positive weak solutions to the power-law thin-film equation \eqref{eq:PDE_power-law} for all flow-behaviour exponents $\alpha > 0$, see Theorem \ref{thm:Local_Ex_PL} below. In the case $\alpha > 1$ of shear-thinning power-law fluids, even global existence of non-negative weak solutions has been established in \cite{ansini_doubly_2004}, using a two-step regularisation scheme, Galerkin approximation and energy/entropy methods. Since the present paper is concerned with stability of positive steady states, we are only interested in positive weak solutions. Therefore, we use a simpler regularisation method that allows us (only) to construct local positive weak solutions, but for all flow-behaviour exponents $\alpha > 0$. These solutions can then be extended to global weak solutions as long as they are close to steady states.

Moreover, again for all $\alpha > 0$, we can characterise positive steady states of the power-law thin-film equation by positive constants, cf. Theorem \ref{thm:char_steady_states} below. As already mentioned, the long-time behaviour of solutions that are initially close to a steady state $\bar{u}_0 = \fint_\Omega u_0\, dx$ depends strongly on the choice of the flow-behaviour exponent $\alpha$, i.e., on the shear-thinning, respectively shear-thickening nature of the fluid. The main result concerning global existence and stability properties of steady states is the following:
\begin{theorem*}
Fix $\alpha > 0$. Then there exists an $\eps > 0$ such that, for all positive initial values $u_0 \in H^1(\Omega)$ with $\|u_0 - \bar{u}_0\|_{H^1(\Omega)} \leq \eps$, problem \eqref{eq:PDE_power-law} possesses at least one global positive weak solution 
\begin{equation*}
    u \in C\bigl([0,\infty);H^1(\Omega)\bigr) \cap L_{\alpha+1,\text{loc}}\bigl((0,\infty);W^3_{\alpha+1,B}(\Omega)\bigr)
    \quad \text{with} \quad
    u_t \in L_{\frac{\alpha+1}{\alpha},\text{loc}}\bigl((0,\infty);(W^1_{\alpha+1,B}(\Omega))^\prime\bigr),
\end{equation*}
satisfying the boundary condition $u_x=0$ on $\partial\Omega$ pointwise for almost every $t \geq 0$.
Moreover, this global solution has the following asymptotic behaviour:
\begin{itemize}
    \item[(i)] In the shear-thickening case $0 < \alpha < 1$, there exists a positive but finite time $0 < t^\ast < \infty$ such that
    \begin{equation*}
        u(t,\cdot) \longrightarrow \bar{u}_0 \text{ in } H^1(\Omega), \text{ as } t \to t^\ast,
        \quad \text{and} \quad
        u(t,x) = \bar{u}_0, 
        \quad
        t \geq t^\ast,\, x \in \Omega.
    \end{equation*}
    \item[(ii)] In the shear-thinning case $1 < \alpha < \infty$, there exists a constant $C > 0$ such that
    \begin{equation*}
        \|u(t) - \bar{u}_0\|_{H^1(\Omega)} \leq
        \frac{C \eps}{\bigl(1 + C \eps^{\alpha-1} t\bigr)^\frac{1}{\alpha-1}},
        \quad
        0 \leq t < \infty.
    \end{equation*}
    \item[(iii)] In the Newtonian case $\alpha=1$, there exist positive constants $C, \gamma > 0$ such that
    \begin{equation*}
        \|u(t) - \bar{u}_0\|_{H^1(\Omega)} \leq C e^{-\gamma t},
        \quad
        0 \leq t < \infty.
    \end{equation*}
\end{itemize}
\end{theorem*}
Note that statement (iii) of this theorem is already well-known \cite{beretta_nonnegative_1995,bertozzi_lubrication_1996} and can even be proved in `better' function spaces with standard theory, see for instance the text books \cite{haragus_local_2011,lunardi_analytic_2012}. Moreover, in the shear-thinning case (ii), convergence to steady states has already been proved in \cite{ansini_doubly_2004} but without rate of convergence.
In the cylindrical Taylor--Couette setting, statement (iii) has first been shown in \cite{pernas_castano_analysis_2020} in the framework of stable center manifolds. Similarly, the results in (i) and (ii) have been obtained in \cite{lienstromberg_analysis_2022} and \cite{lienstromberg_long-time_2022}, also in the cylindrical Taylor--Couette geometry.

Finally, we prove global existence of positive weak solutions to the Ellis-law thin-film equation \eqref{eq:PDE_Ellis} and provide a description of their asymptotic behaviour. For $\alpha \geq 2$, stability and exponential decay to equilibrium can again be obtained by standard techniques \cite{lunardi_analytic_2012,haragus_local_2011}. However, for $1 < \alpha < 2$, these techniques are not applicable since the coefficients of the differential operator are merely Hölder continuous. For this range of flow-behaviour exponents we use energy methods to prove exponential asymptotic stability of steady states in $H^1(\Omega)$.

%--------------------------------------------------
%--------------------------------------------------
\bigskip

\subsection{Shear-Dependent non-Newtonian Fluids} 
Many common liquids and gases, such as water and air, may reasonably be considered Newtonian. However, there is still a multitude of real fluids which are in fact non-Newtonian. Newtonian fluids are characterised by a perfectly linear dependence of the shear stress $\sigma(\epsilon)$ on the local strain rate $\epsilon$, the constant fluid viscosity $\mu > 0$ being the factor of proportionality. In contrast to that, shear-dependent non-Newtonian fluids feature a non-linear relation between the shear-rate and the viscous stress, $\sigma(\epsilon) = \mu(|\epsilon|)\epsilon$, where $\mu(|\epsilon|)$ is the shear-dependent viscosity. That is, these fluids become more solid or more liquid under shear force. In the case in which the fluid viscosity increases with increasing shear rate, the corresponding fluids are called \textbf{shear-thickening}. On the contrary, fluids are called \textbf{shear-thinning} if their viscosity decreases with increasing shear-rate. In this paper, we are concerned with two classes of non-Newtonian fluids, so-called \textbf{power-law fluids} or \textbf{Ostwald--de Waele fluids} and \textbf{Ellis-fluids}. 

%--------------------------------------------------
%--------------------------------------------------
\medskip

\noindent\textbf{\textsc{Power-Law Fluids. }} For \textbf{power-law fluids} or \textbf{Ostwald--de Waele fluids} the constitutive law for the effective fluid viscosity reads
\begin{equation}\label{eq:mu_power-law}
	\mu(|\epsilon|) = \mu_0 |\epsilon|^{\frac{1}{\alpha}-1},
\end{equation}
with a characteristic viscosity $\mu_0 > 0$ and a flow-behaviour exponent $\alpha > 0$. For these fluids, the relation between the local strain and the viscous stress is 
\begin{equation*}
	\sigma(\epsilon) = \mu_0 |\epsilon|^{\frac{1}{\alpha}-1} \epsilon.
\end{equation*}
Note that the corresponding fluid is shear-thickening for flow-behaviour exponents $0 < \alpha < 1$, while it is shear-thinning for $\alpha > 1$. In the case $\alpha=1$, we recover the Newtonian regime $\mu(|\epsilon|)\equiv \mu_0 > 0$ of a constant viscosity. 

However, it is observed in real-world applications (e.g. in polymeric systems) that, at `intermediate' shear rates, fluids behave according to \eqref{eq:mu_power-law}, while the at rather low and/or rather high shear rates, the viscosity approaches a Newtonian plateau. This is obviously not reflected by \eqref{eq:mu_power-law}. 

%--------------------------------------------------
%--------------------------------------------------
\medskip

\noindent\textbf{\textsc{Ellis fluids. }} 
As a second class of shear-dependent non-Newtonian fluids we consider fluids the rheology of which is described by the so-called \textbf{Ellis constitutive law} \cite{weidner_contactline_1994}
\begin{equation}\label{eq:mu_Ellis}
	\frac{1}{\mu(|\epsilon|)}
	=
	\frac{1}{\mu_0} \left(1 + \left|\frac{\sigma(\epsilon)}{\sigma_{1/2}}\right|^{\alpha-1}\right),
	\quad
	\alpha \geq 1,\
	0 < \sigma_{1/2} < \infty,
\end{equation}
where $\sigma(\epsilon) = \mu(|\epsilon|)\epsilon$ is the viscous shear stress.
Here, $\mu_0 > 0$ denotes the viscosity at zero shear stress and $\sigma_{1/2} > 0$ is the viscous shear stress at which the viscosity is reduced to $\mu_0/2$. Thus, for $\alpha > 1$ and $0 < \sigma_{1/2} < \infty$ the Ellis constitutive law describes a shear-thinning behaviour, i.e., the fluid viscosity decreases with increasing shear rate. For $\alpha=1$ or for $\sigma_{1/2}^{\alpha-1} \to \infty$, we recover a Newtonian behaviour. As an advantage over \eqref{eq:mu_power-law}, the Ellis law \eqref{eq:mu_Ellis} has the ability to describe a shear-thinning behaviour for `moderate' shear rates and a Newtonian plateau for rather low shear stresses, since for all $\sigma \in \R$,
\begin{equation*}
	\frac{1}{\mu(|\epsilon|)}
	=
	\frac{1}{\mu_0} \left(1 + \left|\frac{\sigma(\epsilon)}{\sigma_{1/2}}\right|^{\alpha-1}\right)
	\longrightarrow
	\frac{1}{\mu_0},
	\quad
	\text{as}
	\quad
	\sigma_{1/2}^{\alpha-1} \to \infty.
\end{equation*}
For the majority of polymers and polymer solutions the flow-behaviour exponent $\alpha$ in \eqref{eq:mu_Ellis} varies in a range between 1 and 2, see e.g. \cite{bird_dynamics_1987,matsuhisa_analytical_1965}.

A plot of the different constitutive laws for the fluid viscosity (Newtonian fluids, shear-thickening and shear-thinning power-law fluids and Ellis fluids) is offered in Figure \ref{fig:constitutive_laws}.
\begin{center}
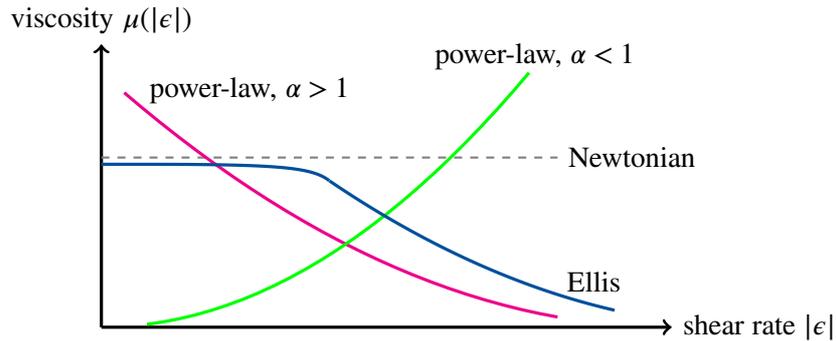
\begin{figure}[h!] 
\begin{tikzpicture}[yscale=1.5, xscale=1.5]
\draw[->, very thick] (0,0)--(5,0) node[right] {shear rate $|\epsilon|$};
\draw[->, very thick] (0,0)-- (0,2.5) node[above] {viscosity $\mu(|\epsilon|)$};
\draw[-, dashed, gray, thick] (0,1.5)--(4,1.5) node[right]{\color{black}{Newtonian}};
\draw[domain=0.2:4,variable=\x,magenta, very thick] plot ({\x},{(1.5-0.3*abs(\x))^2 }); % power-law alpha > 1
\draw[domain=0.4:3.75,variable=\x,green, very thick] plot ({\x},{(0.4*abs(\x))^2}); % power-law alpha < 1
\draw[domain=0:2.0,variable=\x,luh-dark-blue, very thick, samples=200] plot ({\x},{ (1.5/pi)*(pi/2 - rad(atan(-8+(\x-0.2)^3))) });
\draw[domain=2.0:4.5,variable=\x,luh-dark-blue, very thick, samples=200] plot ({\x},{ (sqrt((1.5/pi)*(pi/2 - rad(atan(-8+1.8^3))))-0.3*abs(\x-2.0))^2 });
\node at (3.8,2.4) {power-law, $\alpha < 1$};
\node at (1.3,2.1) {power-law, $\alpha > 1$};
\node at (4.32,0.4) {Ellis};
\end{tikzpicture}
\caption{Constitutive viscosity laws: Newtonian fluid (dashed), shear-thinning power-law fluid (pink), shear-thickening power-law fluid (green) and Ellis-fluid (blue).}
\label{fig:constitutive_laws}
\end{figure}
\end{center}

%-------------------------------------
%-------------------------------------
\medskip

\subsection{Outline of the paper} The structure of the paper is as follows: In Section \ref{sec:Functional_Setting} we introduce the functional setting we will work in. 

In Section \ref{sec:Local_Existence} we prove local existence of positive weak solutions to the power-law thin-film equation and characterise positive steady states by positive constants.

In Section \ref{sec:regularity-estimates} we derive regularity estimates for weak solutions that are valid as long as the solution stays bounded away from zero. More precisely, we prove a {\L}ojasiewicz-Simon-type inequality that estimates the dissipation functional in terms of powers of the energy functional. Moreover, we provide a local $L_1$-in-time estimate for the dissipation functional in terms of the energy at a slightly earlier time.

Section \ref{sec:shear-thickening} is concerned with the dynamic behaviour of solutions to the shear-thickening power-law problem. First, we prove global existence of positive weak solutions for initial film heights that are initially close to a constant in $H^1(\Omega)$. Moreover, we show that these solutions converge to a positive constant in finite time and stay constant for all later times.

Section \ref{sec:shear-thinning} is concerned with the stability properties of solutions to the shear-thinning power-law thin-film equation. As in the shear-thickening case, it is shown that weak solutions exist globally time and stay positive if they are initially close to a steady state. Moreover, these positive global weak solutions are polynomially stable in $H^1(\Omega)$ in the sense that they converge to a steady state (positive constant) at rate $1/t^{1/(\alpha-1)}$, as time tends to infinity.

In Section \ref{sec:Ellis} we study the non-Newtonian thin-film equation that arises when the constitutive law for the fluid viscosity is the Ellis-law. The corresponding Ellis fluids have a Newtonian plateau for small shear rates and behave like a shear-thinning power-law fluid for high shear rates. We observe exponential asymptotic stability of steady states in the $H^1(\Omega)$-norm.

%-----------------------------------------------------
%-----------------------------------------------------
\bigskip

\section{Functional Framework}
\label{sec:Functional_Setting}

In this section we provide the functional setting that will be needed for the study of both the power-law \eqref{eq:PDE_power-law} and Ellis-law \eqref{eq:PDE_Ellis} thin-film equations. 
%We now introduce the functional setting we are working in. 

Throughout this paper, we assume that $\Omega \subset \R$ is a bounded interval. 
For $k \in \N$ and $p \in [1,\infty)$ we denote by $W^k_p(\Omega)$ the usual Sobolev 
spaces with norm
\begin{equation*}
\|v\|_{W^k_p(\Omega)} = \left(\sum_{j=0}^k 
\|\partial^{j} v\|_{L_p(\Omega)}^p\right)^{1/p}.
\end{equation*}
We then define the seminorm
\begin{equation*}
[v]_{W^s_p(\Omega)}
=
\int_{\Omega} \int_{\Omega} \frac{|v(x) - v(z)|^p}{|x-z|^{1+sp}}\, dx\, dz,
\quad
1 \leq p < \infty,\ 0 < s < 1,
\end{equation*}
and introduce the \textbf{fractional Sobolev spaces} by
\begin{equation*}
    W^s_p(\Omega) =
    \left\{v \in W^{[s]}_p(\Omega); \|v\|_{W^s_p(\Omega)} < \infty\right\},
    \quad
    1\leq p < \infty,\ s \in \R_+\setminus\N,
\end{equation*}
where
\begin{equation*}
\|v\|_{W^s_p(\Omega)} = \left(\|v\|_{W^{[s]}_p(\Omega)}^p + [\partial^{[s]}v]_{W^{s-[s]}_p(\Omega)}^p\right)^{1/p},
\quad
1\leq p < \infty,\ s \in \R_+\setminus\N,
\end{equation*}
with $[s]$ denoting the largest integer such that $[s] \leq s$.

We now recall some important properties of these spaces. It is well-known (see, for instance, \cite{triebel_interpolation_1978}) that, 
for $0 \leq s_0 < s_1 < \infty$, $1<p<\infty$, and $0 < \rho < 1$, the space $W^s_p(\Omega)$ with 
$s = (1 - \rho) s_0 + \rho s_1$, is the complex interpolation space between $W^{s_1}_p(\Omega)$ and 
$W^{s_0}_p(\Omega)$, in symbols
\begin{equation*}
W^s_p(\Omega) = [W^{s_0}_p(\Omega),W^{s_1}_p(\Omega)]_{\rho}.
\end{equation*}
In order to take the Neumann-type boundary conditions into account, we further introduce the Banach spaces
\begin{equation*}
W^{4\rho}_{p,B}(\Omega)
=
\begin{cases}
\bigl\{v \in W^{4\rho}_p(\Omega); 
v_x = v_{xxx} = 0 \text{ on } \partial\Omega\bigr\}, 
& 3 + \frac{1}{p} < 4\rho \leq 4,
\\[1ex]
\bigl\{v \in W^{4\rho}_p(\Omega); 
v_x = 0 \text{ on } \partial\Omega\bigr\}, 
& 1+ \frac{1}{p} < 4\rho \leq 3 + \frac{1}{p},
\\[1ex] 
W^{4\rho}_p(\Omega), & 0 \leq 4\rho \leq 1+ \frac{1}{p}.
\end{cases}
\end{equation*}
For $4\rho \in (0,4)\setminus\{1+1/p,3+1/p\}$, the spaces $W^{4\rho}_{p,B}(\Omega)$ are closed
linear subspaces of $W^{4\rho}_{p}(\Omega)$ and satisfy the interpolation property \cite[Theorem 4.3.3]{triebel_interpolation_1978}
\begin{equation}
W^{4\rho}_{p,B}(\Omega) = \bigl(L_p(\Omega),W^4_{p,B}(\Omega)\bigr)_{\rho,p},
\quad
1 < p < \infty. 
\end{equation}
Lastly, we use $W^1_{p,0}(\Omega)$ to denote the space of 
functions belonging to $W^1_{p}(\Omega)$ with zero 
boundary condition.
%-------------------------------------
%-------------------------------------
\bigskip

% Consider the problem
% \begin{equation}\label{eq:QP}
% \begin{cases}
% u_t + \Acal(u) u
% =
% \Fcal(u),
% &
% t > 0,
% \\[1ex]
% u(0)
% =
% u_0,
% \quad
% &
% \end{cases}
% \end{equation}

% where the quasilinear differential operator $\Acal$ is defined as follows. Let $4\sigma > 7/2$ and  $\alpha \in (1,2)$. For $v \in \Ocal_\kappa := \left\{v \in H^{4\sigma}_B(\Omega);\, v(x) > \kappa\, \forall x \in \bar{\Omega}\right\} \subset H^{4\sigma}_B(\Omega)$ we set
% \begin{equation*}
% 	\Acal(v) \in \Lcal\left(H^4_B(\Omega);L_2(\Omega)\right),
% 	\quad
% 	\Acal(v) u := a(v,v_x,v_{xx},v_{xxx}) \partial_x^4 u
% \end{equation*}
% with
% \begin{equation*}
% 	a : C^{\alpha-1}\left(\R^4;\R_+\right)
% 	\quad
% 	\text{ such that }
% 	\quad
% 	a(v,v_x,v_{xx},v_{xxx}) \geq \lambda_\kappa \quad \forall v \in \Ocal_\kappa
% \end{equation*}
% for some ellipticity constant $\lambda > 0$. It follows that
% \begin{equation*}
% 	\Acal \in C^{\alpha-1}\left(H^{4\sigma}_B(\Omega);\Hcal\left(H^4_B(\Omega);L_2(\Omega)\right)\right).
% \end{equation*}
% Moreover, we assume that the right-hand side satisfies
% \begin{equation*}
% 	\Fcal \in C^\alpha\left(H^{4\sigma}_B(\Omega);L_2(\Omega)\right).
% \end{equation*}

% Let $u_\ast \in H^4_B(\Omega)$ such that $u_\ast > \kappa$ be a stationary solution of \eqref{eq:QP}, i.e.
% \begin{equation*}
% 	\Acal(u_\ast) u_\ast = \Fcal(u_\ast).
% \end{equation*}

%-------------------------------------
%-------------------------------------
%-------------------------------------
%-------------------------------------

\section{Local Existence for the Power-Law Thin-Film Equation}
\label{sec:Local_Existence}

In this section we prove local existence of positive weak solutions to the evolution problem
\begin{equation} \label{eq:PDE}
    \begin{cases}
        u_t + \bigl(u^{\alpha+2} |u_{xxx}|^{\alpha-1} u_{xxx}\bigr)_x
        =
        0,
        & t>0,\ x \in \Omega,
        \\
        u_x(t,x) = u_{xxx}(t,x) = 0, 
        &
        t > 0,\ x \in \partial\Omega,
        \\
        u(0,x) = u_0(x),
        &
        x \in \Omega,
    \end{cases}
\end{equation}
for flow-behaviour exponents $\alpha >0$, i.e., for both shear-thinning $(\alpha > 1)$ and shear-thickening ($\alpha < 1$) power-law fluids. Moreover, we characterise the positive steady states of \eqref{eq:PDE} by positive constants (flat films of positive height). 

Our analysis strongly relies on an energy-dissipation estimate for the \textbf{energy functional} 
\begin{equation*}
    E[u] = \frac{1}{2} \int_\Omega |u_x|^2\,dx.
\end{equation*}
Formally testing the equation with the second derivative $u_{xx}$, one finds that $E[u](t)$ decreases along solutions to \eqref{eq:PDE}. More precisely, solutions $u$ to \eqref{eq:PDE} satisfy
\begin{equation*}
    \frac{d}{dt} E[u](t) = -D[u](t) = -\int_\Omega u^{\alpha+2} |u_{xxx}|^{\alpha+1}\, dx. 
\end{equation*}
We call $D[\cdot]$ the \textbf{dissipation functional}.

For the purpose of local existence, we introduce in Section \ref{ssec:Existence_PL_reg} a regularised version of \eqref{eq:PDE} that removes the degeneracy in the third derivative $u_{xxx}$. For the regularised problem we apply standard parabolic theory in order to prove existence of positive strong solutions, emanating from positive initial values.
In Section \ref{ssec:Existence_PL} we provide uniform a-priori bounds for the solutions to the regularised problem and pass to the limit of a vanishing regularisation parameter in order to obtain local existence of positive weak solutions to the original problem \eqref{eq:PDE}. 

Note that for $\alpha > 1$ (shear-thinning fluids) existence of global non-negative weak solutions is already proved in \cite{ansini_doubly_2004}, where the authors use a more involved regularisation scheme. However, in the present paper we are only interested in positive solutions, but for all flow-behaviour exponents $\alpha > 0$.

%-------------------------------------
%-------------------------------------

In order to simplify notation, we introduce, for a fixed $\alpha > 0$, the function
\begin{equation*}
    \psi\colon \R \to \R,
    \quad
    s \mapsto \psi(s) = |s|^{\alpha-1} s,
\end{equation*}
and rewrite the partial differential equation $\eqref{eq:PDE}_1$ as
\begin{equation*}
    u_t + \bigl(u^{\alpha+2} \psi(u_{xxx})\bigr)_x = 0,
    \quad
    t > 0,\ x \in \Omega.
\end{equation*}
Note that if $\alpha \geq 1$, then $\psi \in C^1(\R)$ with $\psi^\prime(s) = \alpha |s|^{\alpha-1}$. For $\alpha < 1$ the function $\psi$ is only $\alpha$-H\"older-continuous.

%-------------------------------------
%-------------------------------------

\begin{definition}\label{def:weak_sol}
For a given $T > 0$ and initial value $u_0 \in H^1(\Omega)$, a weak solution to \eqref{eq:PDE} is defined as a function
\begin{equation*}
    u \in C\bigl([0,T];H^1(\Omega)\bigr) 
    \cap
    L_{\alpha+1}\bigl((0,T);W^3_{\alpha+1,B}(\Omega)\bigr)
    \quad
    \text{with}
    \quad
    u_t \in L_\frac{\alpha+1}{\alpha}\bigl((0,T);(W^1_{\alpha+1,B}(\Omega))'\bigr)
\end{equation*} 
that has the following properties:
\begin{itemize}
    \item[(i)] (Weak formulation) $u$ satisfies the differential equation  $\eqref{eq:PDE}_1$ in the weak sense, i.e., 
    \begin{equation*}
        \int_0^T \langle u_t,\phi\rangle_{W^1_{\alpha+1}(\Omega)}\, dt
        =
        \int_0^T \int_\Omega u^{\alpha+2} \psi(u_{xxx})\, \phi_x\, dx\, dt
    \end{equation*}
    for all test functions $\phi \in L_{\alpha+1}\bigl((0,T);W^1_{\alpha+1,B}(\Omega)\bigr)$.
    \item[(ii)] (Initial and boundary values) $u$ satisfies the contact angle condition \(u_x=0\) on \(\partial\Omega\) and the initial condition $\eqref{eq:PDE}_3$ pointwise.
\end{itemize}
\end{definition}

%-------------------------------------
%-------------------------------------

The following theorem contains the main result of this section.

%-------------------------------------
%-------------------------------------

\begin{theorem}[Local existence of positive weak solutions] \label{thm:Local_Ex_PL}
Given a positive initial value $u_0 \in W^{4\rho}_{\alpha+1,B}(\Omega),\ 4\rho > 3+1/(\alpha+1)$, with
$u_0(x) > 0,\ x \in \bar{\Omega}$,
% \begin{equation*}
%     \frac{1}{|\Omega|} \int_\Omega u_0(x)\, dx = \bar{u}_0
%     \quad \text{and} \quad
%     \|u_0 - \bar{u}_0\|_{H^1(\Omega)} \leq \eps
% \end{equation*}
there exists a time $T > 0$ such that problem \eqref{eq:PDE} admits at least one positive weak solution \begin{equation*}
    u \in C\bigl([0,T];H^1(\Omega)\bigr) 
    \cap
    L_{\alpha+1}\bigl((0,T);W^3_{\alpha+1,B}(\Omega)\bigr)
    \quad \text{with} \quad
    u_t \in L_\frac{\alpha+1}{\alpha}\bigl((0,T);(W^1_{\alpha+1,B}(\Omega))'\bigr)
\end{equation*}
on $(0,T)$ in the sense of Definition \ref{def:weak_sol}. Moreover, such a solution has the following properties:
\begin{itemize}
    \item[(i)] (Positivity) $u$ is bounded away from zero
    \begin{equation*}
        0
        < C_T \leq
        u(t,x),
        \quad
        0 \leq t \leq T,\ x \in \bar{\Omega}.
    \end{equation*}
    \item[(ii)] (Conservation of mass) $u$ conserves its mass in the sense that
        \begin{equation*}
            \|u(t)\|_{L_1(\Omega)} 
            =
            \|u_0\|_{L_1(\Omega)},
            \quad 0 \leq t \leq T.
        \end{equation*}
    \item[(iii)] (Energy-dissipation identity) Energy is dissipated along solutions
        \begin{equation}\label{eq:energy-dissipation}
            E[u](t) + \int_0^t D[u](s) \, ds= E[u_0]
        \end{equation}
    for almost every $t\in [0,T]$. 
\end{itemize}
\end{theorem}

%-------------------------------------
%-------------------------------------

Observe that due to the positivity of a solution $u$ to \eqref{eq:PDE} we have
\begin{equation*}
    \int_\Omega u(t,x)\, dx
    =
    \|u(t)\|_{L_1(\Omega)} 
    =
    \|u_0\|_{L_1(\Omega)},
    \quad 0 \leq t \leq T.
\end{equation*}

%-------------------------------------
%-------------------------------------

\begin{remark}\label{Rem:H1-initial}
    In fact, the above theorem holds true for initial values $u_0 \in H^1(\Omega)$. We choose $u_0$ in the smaller space $W^{4\rho}_{\alpha+1}(\Omega)$ since the solutions $u$ to the original problem are constructed as accumulation points of \textbf{strong} solutions $u^\sigma$ to a regularised problem, not only as functions satisfying a suitable weak formulation. In order to apply semigroup theory, we require the initial value to satisfy $u_0 \in W^{4\rho}_{\alpha+1}(\Omega)$. That $u_0 \in H^1(\Omega)$ is enough can be seen by replacing $u_0$ by $u_0^\sigma \in W^{4\rho}_{\alpha+1}(\Omega)$ with 
    \begin{equation*}
        u_0^\sigma(x) > 0,\ x \in \bar{\Omega},
        \quad 
        \bar{u}_0^\sigma = \bar{u}_0 = \fint_\Omega u_0\, dx
        \quad \text{and} \quad 
        u_0^\sigma \longrightarrow u_0 \quad \text{strongly in } H^1(\Omega), \text{ as } \sigma \searrow 0.
    \end{equation*}
    This can for instance be obtained by a symmetric extension of the initial value $u_0 \in H^1(\Omega)$ at the lateral boundaries and mollification.
\end{remark}

%-----------------------------------------------------
%-----------------------------------------------------

\begin{remark}\label{Rem:extension}
Given a positive weak solution $u \in C\bigl([0,T];H^1(\Omega)\bigr) 
\cap L_{\alpha+1}\bigl((0,T);W^3_{\alpha+1,B}(\Omega)\bigr)$ to \eqref{eq:PDE} as obtained in Theorem \ref{thm:Local_Ex_PL}, we may extend it beyond time \(T\) by restarting the equation with initial datum \(u(T)\) and using that \(u(T,x)>0\) for all \(x\in \bar{\Omega}\) and Remark \ref{Rem:H1-initial}. In fact, in this way we can construct a weak solution to \eqref{eq:PDE} in the sense of Definition \ref{def:weak_sol} up to a time \(T_*>0\) at which \(u(T_*,x) = 0\) for some \(x\in \bar{\Omega}\). Note though, that the solutions in Theorem \ref{thm:Local_Ex_PL} are not unique, so that the `maximal' time \(T_*\) of existence of positive solutions is not unique. 
\end{remark}

%-----------------------------------------------------
%-----------------------------------------------------
\bigskip

\noindent\textbf{\textsc{Positive Steady states of \eqref{eq:PDE}. }}
We are interested in the stability properties of steady-state solutions to \eqref{eq:PDE}, i.e., in functions $u^\ast \in W^3_{\alpha+1,B}(\Omega)$ that solve the ordinary differential equation
\begin{equation} \label{eq:ODE}
    U^{\alpha+2} |U^{\prime\prime\prime}|^{\alpha-1} U^{\prime\prime\prime}
    =
    0,
    \quad
    x \in \Omega.
\end{equation}
In physical parlance, \eqref{eq:ODE} says that there is no flux of the fluid through the boundaries of the interval. Positive steady states of \eqref{eq:PDE} may be easily characterised by the following theorem.

%-----------------------------------------------------
%-----------------------------------------------------
\begin{theorem}[Characterisation of positive steady states] \label{thm:char_steady_states}
A function $u \in W^3_{\alpha+1,B}(\Omega)$ is a positive steady-state solution of \eqref{eq:PDE} if and only if $u \equiv u^\ast \in \R_{>0}$ is given by a positive constant.
\end{theorem}

%-----------------------------------------------------
%-----------------------------------------------------

\begin{proof}
\noindent\textbf{(i)} Let $u \equiv u^\ast \in \R_{>0}$. Then $u^\ast \in W^{3}_{\alpha+1,B}(\Omega)$ clearly satisfies the ODE \eqref{eq:ODE}.

\noindent\textbf{(ii)} Let $u = u^\ast \in W^3_{\alpha+1,B}(\Omega)$ be an arbitrary positive steady-state solution of \eqref{eq:PDE}, i.e., a solution to the ODE \eqref{eq:ODE}. Then $u^\ast$ satisfies
\begin{equation*}
    0 
    =
    \frac{d}{dt} E[u^\ast]
    =
    -D[u^\ast]
    =
    -\int_\Omega |u^\ast|^{\alpha+2} |u^\ast_{xxx}|^{\alpha+1}\, dx.
\end{equation*}
Since the integrand on right-hand side of this equation is non-negative and $u^\ast(x) > 0,\, x \in \bar{\Omega}$, it follows that $u^\ast_{xxx} \equiv 0$ on $\bar{\Omega}$. Consequently, $u^\ast_{xx}$ is constant and this in turn implies that $u^\ast_x$ is linear. Taking the Neumann boundary conditions into account, we find that $u^\ast$ must be constant.
\end{proof}

%-----------------------------------------------------
%-----------------------------------------------------
\bigskip

\subsection{Local Existence of Positive Solutions to the Regularised Problem and Uniform A-Priori Bounds} \label{ssec:Existence_PL_reg}

In order to handle the difficulties caused by the doubly nonlinear and doubly degenerate nature of the evolution problem \eqref{eq:PDE}, we introduce, for a fixed regularisation parameter $\sigma \in (0,1)$ and all $s\in\R$, the smooth function 
\begin{equation*}
    \psi_\sigma(s) = (s^2 + \sigma^2)^\frac{\alpha-1}{2} s,
    \quad 
    s \in \R,
\end{equation*}
and substitute the nonlinear term $\psi(u_{xxx})$ in \eqref{eq:PDE} accordingly. The \textbf{regularised problem} corresponding to \eqref{eq:PDE} then reads
\begin{equation}\label{eq:PDE_reg}\tag{$P_\sigma$}
    \begin{cases}
        u^\sigma_t + \bigl((u^\sigma)^{\alpha+2} \psi_\sigma(u^\sigma_{xxx})\bigr)_x
        =
        0,
        &
        t > 0,\, x \in \Omega,
        \\
        u^\sigma_x(t,x) = u^\sigma_{xxx}(t,x) = 0,
        &
        t > 0,\, x \in \partial\Omega,
        \\
        u^\sigma(0,x)
        =
        u_0(x), &
        x \in \Omega.
    \end{cases}
\end{equation}
It follows from standard parabolic theory \cite{amann_nonhomogeneous_1993,eidelman_parabolic_1969,lienstromberg_local_2020} that the regularised problem \eqref{eq:PDE_reg} possesses, for each fixed $\sigma \in (0,1)$ and suitable initial data, a unique maximal strong solution $u^\sigma$. This is the content of Theorem \ref{thm:Local_Ex_PL_reg} below. Moreover, in Lemma \ref{lem:uniform_bounds} below, we provide a-priori bounds for the strong solution that are uniform in the regularisation parameter $\sigma > 0$. First, though, we define what we mean by a maximal strong solution to \eqref{eq:PDE_reg}. 

%-----------------------------------------------------
%-----------------------------------------------------

\begin{definition}
Fix $\alpha > 0$ and $\sigma \in (0,1)$. Let $1 < p < \infty$. Given a positive initial value $u_0 \in L_p(\Omega)$, we call a function $u\colon [0,T_u) \to L_p(\Omega)$ with $u(t,x) > 0$ for $t \in [0,T_u)$ and $x \in \bar{\Omega}$ a \textbf{maximal positive strong solution} to \eqref{eq:PDE_reg} on $[0,T_u)$ in $L_p(\Omega)$ if the following conditions are satisfied:\\[-0.3cm]
\begin{itemize}
    \item[(i)] $u \in C\bigl([0,T_u);L_p(\Omega)\bigr) \cap C^1\bigl((0,T_u);L_p(\Omega)\bigr)$;\\[-0.3cm]
    \item[(ii)] $u(0) = u_0 \in L_p(\Omega)$ and $u(t) \in W^4_{p,B}(\Omega)$ for all $t \in (0,T_u)$;\\[-0.3cm]
    \item[(iii)] (Positivity) $u(t,x) > 0$ for $t \in [0,T_u)$ and $x \in \bar{\Omega}$;\\[-0.3cm]
    \item[(iv)] $u$ satisfies the differential equation $\eqref{eq:PDE_reg}_1$ pointwise;\\[-0.3cm]
    \item[(v)] (Maximality) There is no other solution $v$ on $[0,T_v)$ with $T_u < T_v$.
\end{itemize}
\end{definition}

%-----------------------------------------------------
%-----------------------------------------------------
Clearly, solutions to \eqref{eq:PDE_reg}, as obtained in the following theorem, do also dissipate energy. We therefore introduce the notation
\begin{equation*}
    D_\sigma[u^\sigma](t) = \int_\Omega (u^\sigma)^{\alpha+2} |u^\sigma_{xxx}|^{\alpha+1}\, dx
\end{equation*}
for the dissipation functional corresponding to the energy functional $E[\cdot]$ and the regularised equation \eqref{eq:PDE_reg}.

\begin{theorem}[Local existence for \eqref{eq:PDE_reg}] \label{thm:Local_Ex_PL_reg}
Fix $\alpha > 0$ and $\sigma \in (0,1)$. Let $1/(\alpha +1) < s < r < 1$. Moreover, let $\theta = \frac{3+s}{4}$ and 
$\rho = \frac{3+r}{4}$. Then, given an initial film height 
$u_0 \in W^{4\rho}_{\alpha+1,B}(\Omega)$ such that $u_0(x) > 0\,$ for all $x \in \bar{\Omega}$, problem \eqref{eq:PDE_reg} possesses a unique maximal solution
\begin{equation*}
	u^\sigma \in C\bigl([0,T_\sigma);W^{4\rho}_{\alpha+1,B}(\Omega)\bigr)
	\cap
	C^{\rho}\bigl([0,T_\sigma);L_{\alpha+1}(\Omega)\bigr)
	\cap
	C\bigl((0,T_\sigma);W^4_{\alpha+1,B}(\Omega)\bigr)
	\cap
	C^1\bigl((0,T_\sigma);L_{\alpha+1}(\Omega)\bigr).
\end{equation*}
Moreover, the solution enjoys the following properties.
\begin{itemize}
    \item[(i)] (Positivity) $u^\sigma$ is positive
        \begin{equation*}
        	u^\sigma(t,x) > 0,
            \quad
            0 \leq t < T_\sigma,\, x\in \bar{\Omega}.
            \end{equation*}
    \item[(ii)] (Conservation of mass) $u^\sigma$ conserves its mass in the sense that
        \begin{equation}\label{eq:cons-mass-sigma}
            \|u^{\sigma}(t)\|_{L_1(\Omega)} 
            =
            \|u_0\|_{L_1(\Omega)},
            \quad 0 \leq t < T_\sigma.
        \end{equation}
    \item[(iii)] (Energy-dissipation identity) $u^\sigma$ satisfies the energy-dissipation identity
    \begin{equation}\label{eq:energy-diss-sigma}
        E[u^\sigma](t) + \int_0^t D_{\sigma}[u^{\sigma}](s) 
        =
        E[u_0],
        \quad
        0 \leq t < T_\sigma.
    \end{equation}
    \item[(iv)] (Maximal time of existence) Suppose that $T_\sigma < \infty$. Then
        \begin{equation*}
            \liminf_{t \nearrow T_\sigma} \frac{1}{\min_{x \in \bar{\Omega}} u^\sigma(t)} 
            + 
            \|u^\sigma(t)\|_{W^{4\gamma}_{\alpha+1,B}(\Omega)} 
            = 
            \infty
        \end{equation*}
    for all $\gamma \in (\theta,1]$.
\end{itemize}
\end{theorem}

%-----------------------------------------------------
%-----------------------------------------------------

\begin{proof}
\noindent\textbf{(i) Local existence, uniqueness and positivity. } In order to prove local existence and uniqueness of a strong solution we apply \cite[Theorem 4.2]{lienstromberg_local_2020}. To this end, we verify that \eqref{eq:PDE_reg} fits into the corresponding abstract functional setting. 
% Indeed, for $1 < p < \infty$ and $s > 1/p$ we define the open subset 
% \begin{equation*}
%     \Scal = \left\{v \in W^{3+s}_{p,B}(\Omega);\, v(x) > 0,\ x \in \bar{\Omega}\right\}
% \end{equation*}
% of $W^{3+s}_{p,B}(\Omega)$. 
Moreover, after rewriting \eqref{eq:PDE_reg} in non-divergence form, we define for $v(t) \in W^{4\theta}_{\alpha+1,B}(\Omega)$ with $\theta=(3+s)/4$ such that $v(x) > 0,\ x \in \bar{\Omega}$, the linear differential operator $\Acal(v(t)) \in \Lcal\bigl(W^4_{\alpha+1,B}(\Omega);L_{\alpha+1}(\Omega)\bigr)$ of fourth order by
\begin{equation*}
    \Acal(v(t)) u^\sigma
    =
    A(v(t)) \partial_x^4 u^\sigma
    \quad
    \text{with}
    \quad
    A(v(t)) = v^{\alpha+2} \psi_\sigma^\prime(v_{xxx}),
\end{equation*}
where 
\begin{equation*}
    \psi_\sigma^\prime(s) 
    = 
    (\alpha-1) (s^2 + \sigma^2)^\frac{\alpha-3}{2} s^2 + (s^2 + \sigma^2)^\frac{\alpha-1}{2}
    =
    \alpha (s^2 + \sigma^2)^\frac{\alpha-1}{2}
    -
    \sigma^2 (\alpha-1) (s^2 + \sigma^2)^\frac{\alpha-3}{2},
    \quad s \in \R.
\end{equation*}
Note that for positive $\sigma \in (0,1)$ we have $\psi_\sigma^\prime(s) > C_{\sigma,\alpha} > 0$ for all $s\in \R$ and all fixed $\alpha > 0$.
Moreover, we introduce the right-hand side
\begin{equation*}
    \Fcal(v(t)) 
    =
    -(\alpha+2) v^{\alpha+1} v_x\, \psi_\sigma(v_{xxx})
\end{equation*}
and perceive \eqref{eq:PDE_reg} as an abstract quasilinear Cauchy problem
\begin{equation*}
    \begin{cases}
        \dot{u^\sigma} + \Acal(u^\sigma)u^\sigma = \Fcal(u^\sigma), \quad t > 0,
        \\
        u^\sigma(0) = u_0.
    \end{cases}
\end{equation*}
Note that the Neumann-type boundary conditions $\eqref{eq:PDE_reg}_2$ are incorporated in the domain $W^4_{\alpha+1,B}(\Omega)$ of the operator $\Acal(v(t))$. Due to the smoothness of $\psi_\sigma$ the maps
\begin{equation*}
    \Acal\colon W^{3+s}_{\alpha+1,B}(\Omega) \longrightarrow
    \Lcal\bigl(W^4_{\alpha+1,B}(\Omega);L_{\alpha+1}(\Omega)\bigr)
    \quad \text{and} \quad
    \Fcal\colon W^{3+s}_{\alpha+1,B}(\Omega) \longrightarrow L_{\alpha+1}(\Omega)
\end{equation*}
are, for all $\alpha > 0$, locally Lipschitz continuous. In order to guarantee parabolicity, we extend the differential operator $\Acal$ to the differential operator 
\begin{equation*}
    \bar{\Acal}_\eps(v(t)) \in \Lcal\bigl(W^4_{\alpha+1,B}(\Omega);L_{\alpha+1}(\Omega)\bigr),
    \quad
    \bar{\Acal}_\eps(v(t))u^\sigma = \bar{A}_{\eps}(v(t)) \partial_x^4 u^\sigma,
\end{equation*}
where
\begin{equation*}
    \bar{A}_{\eps}(v(t)) = \max\left\{v^{\alpha+2}_+ \psi_\sigma^\prime(v_{xxx}),\eps/2\right\}
\end{equation*}
and $v_+ = \max\{v,0\}$. Following the lines of \cite[Chapter 5]{lienstromberg_local_2020}, we study the extended parabolic problem with $\bar{\Acal}_\eps$ instead of $\Acal$ and show that the corresponding local positive solution $u^\sigma=u^\sigma(\eps)$ also solves the non-extended problem \eqref{eq:PDE_reg} for a short but strictly positive time.
More precisely, the extended regularised problem is, for each fixed $\sigma,\eps \in (0,1)$, parabolic in the sense that $\bar{\Acal}_{\eps}(v(t))$ generates an analytic semigroup on $L_{\alpha+1}(\Omega)$. Indeed, due to the embedding $W^{3+s}_{\alpha+1,B}(\Omega) \hookrightarrow C^3(\bar{\Omega})$ and the positivity of $\sigma,\eps > 0$, we have that $\bar{A}_\eps(v(t,\cdot)) \in C(\bar{\Omega})$.
Moreover, the principal symbol $a_{\eps}(x,\xi)$ satisfies
\begin{equation*}
    \text{Re}(a_\eps(x,\xi)\eta | \eta)
    \geq
    C_{\sigma,\alpha,\eps}(i\xi)^4 \eta^2 > 0,
    \quad
    (x,\xi) \in \bar{\Omega}\times \{-1,1\},\, \eta \in \R\setminus\{0\},
\end{equation*}
for a positive constant $C_{\sigma,\alpha,\eps} > 0$.
Consequently, $\bar{\Acal}_\eps(v(t))$, together with the Neumann-type boundary conditions, is normally elliptic in the sense of \cite[Example 4.3(d)]{amann_nonhomogeneous_1993} and we can apply \cite[Theorem  4.1 and Remark 4.2(b)]{amann_nonhomogeneous_1993} to conclude that $\bar{\Acal}_\eps(v(t))$ generates an analytic semigroup on $L_{\alpha+1}(\Omega)$. Thus, we are in the abstract setting of \cite[Theorem 4.2]{lienstromberg_local_2020} which yields existence and uniqueness of a local positive strong solution to the extended problem in $L_{\alpha+1}(\Omega)$. On a potentially smaller time interval, this solution $u^\sigma=u^\sigma(\eps)$ is, for $\eps$ small enough, also a local positive strong solution to \eqref{eq:PDE_reg}, see step (iii) in the proof of  \cite[Theorem 5.1]{lienstromberg_local_2020}.

\noindent\textbf{(ii) Conservation of mass. } This follows by testing the regularised partial differential equation $\eqref{eq:PDE_reg}_1$ with the constant function $\varphi \equiv 1$, integration by parts and using the Neumann boundary conditions $\eqref{eq:PDE_reg}_2$.

\noindent\textbf{(iii) Energy-dissipation identity. } Since the solution obtained in step (i) enjoys the regularity
\begin{equation*}
    u_x^\sigma \in C\bigl((0,T);W^1_{\alpha+1,0}(\Omega)\bigr) \cap C^1\bigl((0,T);(W^1_{\alpha+1,0}(\Omega))^\prime\bigr),
\end{equation*}
we may apply 
\cite[Proposition 6.1]{lienstromberg_local_2020} in order to guarantee that the expression 
\begin{equation*}
    \frac{d}{dt} E[u^\sigma](t)
    =
    \int_\Omega u^\sigma_{xt} u^\sigma_x\, dx
    =
    -\int_\Omega |u^\sigma|^{\alpha+2} \bigl(|u^\sigma_{xxx}|^2 + \sigma^2\bigr)^\frac{\alpha-1}{2} |u^\sigma_{xxx}|^2\, dx
    =
    -D_\sigma[u^\sigma](t)
\end{equation*}
is well-defined for all $t \in (0,T)$. Integrating with respect to time gives the energy-dissipation identity.

\noindent\textbf{(iv) Maximal time of existence. } 
Using the notation introduced in step (i), this result is a minor adaptation of \cite[Theorem 7.1]{lienstromberg_local_2020}.

% We prove the contraposition and assume  that $T_\sigma \leq \tau$. Then we have
% \begin{equation*}
%     \liminf_{t\to T_\sigma^-} \min_{\theta \in S^1} h^\sigma(t,\theta) = 0
%     \quad \text{or} \quad
%     \limsup_{t\to T_\sigma^-}\, \max_{\theta\in S^1} h^\sigma(t,\theta) = \infty,
% \end{equation*}
% which contradicts the assumption on $\tau$.
\end{proof}

%-----------------------------------------------------
%-----------------------------------------------------

In order to prove the local-existence result for the original problem (Theorem \ref{thm:Local_Ex_PL}), we need suitable uniform (in $\sigma$) a-priori estimates for the solution to \eqref{eq:PDE_reg} as given in the following lemma.

%-----------------------------------------------------
%-----------------------------------------------------

\begin{lemma}[Uniform bounds] \label{lem:uniform_bounds}
Let $u^\sigma$ be the maximal solution to \eqref{eq:PDE_reg} for a fixed $\sigma \in (0,1)$ and an initial value $u_0 \in W^{4\rho}_{\alpha+1,B}(\Omega)$ such that $u_0(x) > 0\,$ for all $x \in \bar{\Omega}$.
Then the following holds true. There is \(T>0\) such that the family
$(u^\sigma)_\sigma$ has the following properties:\\[-0.3cm]
\begin{itemize}
	\item[(i)] $(u^\sigma)_\sigma$ is uniformly bounded in $L_\infty\bigl((0,T);H^1(\Omega)\bigr)$;\\[-0.3cm] %\cap L_{\alpha+1}\bigl((0,T);W^3_{\alpha+1}(\Omega)\bigr)$;
	\item[(ii)] $\bigl(\left|u^\sigma\right|^{\alpha+2} \psi_\sigma(u^\sigma_{xxx})\bigr)_\sigma$ is uniformly bounded in $L_\frac{\alpha+1}{\alpha}\bigl((0,T)\times \Omega\bigr)$;\\[-0.3cm]
	\item[(iii)] $(u^\sigma_t)_\sigma$ is uniformly bounded in $L_\frac{\alpha+1}{\alpha}\bigl((0,T);(W^1_{\alpha+1,B}(\Omega))'\bigr)$;\\[-0.3cm]
	\item[(iv)] $(u^\sigma_{xxx})_\sigma$ is uniformly bounded in $L_{\alpha+1}\bigl((0,T)\times \Omega\bigr)$;\\[-0.3cm]
	\item[(v)] $(u^\sigma)_\sigma$ is uniformly bounded in $L_{\alpha+1}\bigl((0,T);W^3_{\alpha+1,B}(\Omega)\bigr)$;\\[-0.3cm]
	\item[(vi)] $((u^\sigma_x)_t)_\sigma$ is uniformly bounded in $L_\frac{\alpha+1}{\alpha}\bigl((0,T);\bigl(W^1_{\alpha+1,0}(\Omega)\cap W^2_{\alpha+1}(\Omega)\bigr)'\bigr)$.
\end{itemize}
\end{lemma}

\begin{proof}
Note that once we have proved items (i) and (iii), the Aubin--Lions--Simon Theorem \cite{simon_compact_1986} implies that the family $(u^{\sigma})_{\sigma}$ is equicontinuous. Hence, we may choose \(T>0\) such that \(u^{\sigma}\) is bounded away uniformly from zero on the interval \([0,T]\).

Within this proof, $C > 0$ denotes a positive constant, possibly depending on $\alpha$, $\Omega$, and $\|u_0\|_{W^{4\rho}_{\alpha+1}(\Omega)}$, but independent of $\sigma$. 

\noindent\textbf{(i)} Since
    \begin{equation*}
        D_{\sigma}[u^{\sigma}](t) = \int_{\Omega} (u^{\sigma})^{\alpha+2} \psi_{\sigma}(u^{\sigma}_{xxx}) \, u^{\sigma}_{xxx} \, dx \geq 0,
        \quad  t \in [0,T_{\sigma}),
    \end{equation*}
    we have
    \begin{equation}
    \label{eq:unifen}
        E[u^{\sigma}](t) = \frac{1}{2} \|u_x^{\sigma}( t) \|_{L_2(\Omega)}^2 \leq E[u_0], \quad  t \in [0,T_{\sigma}).
    \end{equation}
    Using Poincaré's inequality and \eqref{eq:cons-mass-sigma}, 
    we obtain for $t \in [0,T_{\sigma})$
    \begin{equation*}
    \|u^{\sigma}(t)\|_{L_2(\Omega)} \leq 
    \|u^{\sigma}(t)- \bar{u}^{\sigma}(t)\|_{L_2(\Omega} + \|\bar{u}^{\sigma}(t)\|_{L_2(\Omega)} \leq C\|u^{\sigma}_x(t)\|_{L_2(\Omega)} + \|\bar{u}_0\|_{L_2(\Omega)},
    \end{equation*}
    which, together with \eqref{eq:unifen}, yields 
    \begin{equation*}
    \sup \limits_{0\leq t\leq T_{\sigma}} \|u^{\sigma}(t)\|_{H^1(\Omega)} \leq C\bigl( \|\bar{u}_0\|_{L_2(\Omega)} + E[u_0]^{1/2} \bigr).
    \end{equation*}
    Hence, \((u^{\sigma})_{\sigma}\) is uniformly bounded in \(L_{\infty}\bigl((0,T_{\sigma});H^1(\Omega)\bigr)\).
    
    \noindent\textbf{(ii)} First we consider the case \(0<\alpha <1\). 
    Observe that 
    \begin{align*}
        \bigl\| |u^{\sigma}|^{\alpha+2}\psi_{\sigma}(u^{\sigma}_{xxx}) \bigr\|_{L_{\frac{\alpha+1}{\alpha}}((0,T_{\sigma})\times \Omega)}^{\frac{\alpha+1}{\alpha}} & = \int_0^{T_{\sigma}} \int_{\Omega} |u^{\sigma}|^{(\alpha+2)\frac{\alpha+1}{\alpha}} \bigl(|u^{\sigma}_{xxx}|^2 + \sigma^2\bigr)^{\frac{\alpha-1}{2}\frac{\alpha+1}{\alpha}}|u^{\sigma}_{xxx}|^{\frac{\alpha+1}{\alpha}} \, dx \, dt \\
        & = \int_0^{T_{\sigma}} \int_{\Omega} |u^{\sigma}|^{(\alpha+2)\frac{\alpha+1}{\alpha}} \bigl(|u^{\sigma}_{xxx}|^2 + \sigma^2\bigr)^{\frac{\alpha-1}{2}\frac{\alpha+1}{\alpha}}|u^{\sigma}_{xxx}|^{\frac{1-\alpha}{\alpha}}|u^{\sigma}_{xxx}|^{2} \, dx \, dt.
    \end{align*}
    Using that \(\frac{1-\alpha}{\alpha} >0\), we get the pointwise estimate \(|u^{\sigma}_{xxx}|^{\frac{1-\alpha}{\alpha}} \leq \bigl( |u^{\sigma}_{xxx}|^2 + \sigma^2\bigr)^{\frac{1-\alpha}{2\alpha}}\). Furthermore,  in view of (i) and $H^1(\Omega) \hookrightarrow L_\infty (\Omega)$, we find that  \((u^{\sigma})_{\sigma}\) is uniformly bounded in \(L_{\infty}\bigl((0,T_{\sigma}); L_\infty(\Omega)\bigr)\). Combining this, we obtain the estimate
    \begin{align*}
         \int_0^{T_{\sigma}} \int_{\Omega} & |u^{\sigma}|^{(\alpha+2)\frac{\alpha+1}{\alpha}} \bigl(|u^{\sigma}_{xxx}|^2 + \sigma^2\bigr)^{\frac{\alpha-1}{2}\frac{\alpha+1}{\alpha}}|u^{\sigma}_{xxx}|^{\frac{1-\alpha}{\alpha}}|u^{\sigma}_{xxx}|^{2} \, dx \, dt \\
        \leq & \int_0^{T_{\sigma}} \int_{\Omega} |u^{\sigma}|^{(\alpha+2)\frac{\alpha+1}{\alpha}} \bigl(|u^{\sigma}_{xxx}|^2 + \sigma^2\bigr)^{\frac{\alpha-1}{2}}|u^{\sigma}_{xxx}|^{2} \, dx \, dt \\
        \leq & C \int_0^{T_\sigma} D_\sigma[u^{\sigma}](t) \, dt \\
        \leq & C E[u_0],
    \end{align*}
    where the last step is due to \eqref{eq:energy-diss-sigma}.
    \noindent In the case \(1<\alpha <\infty\), we have to use a different argument. Note that by (i) and \eqref{eq:energy-diss-sigma}, we have 
    \begin{align*}
        & \bigl\| |u^{\sigma}|^{\alpha+2}\psi_{\sigma}(u^{\sigma}_{xxx}) \bigr\|_{L_{\frac{\alpha+1}{\alpha}}((0,T_{\sigma})\times \Omega)}^{\frac{\alpha+1}{\alpha}}\\
        = & \int_0^{T_{\sigma}} \int_{\Omega} |u^{\sigma}|^{(\alpha+2)\frac{\alpha+1}{\alpha}} \bigl(|u^{\sigma}_{xxx}|^2 + \sigma^2\bigr)^{\frac{\alpha-1}{2}\frac{\alpha+1}{\alpha}}|u^{\sigma}_{xxx}|^{\frac{\alpha+1}{\alpha}} \, dx \, dt \\
        \leq & C \int_{\{|u^{\sigma}_{xxx}| \leq \sigma\}} \bigl(|u^{\sigma}_{xxx}|^2 + \sigma^2\bigr)^{\frac{\alpha-1}{2}\frac{\alpha+1}{\alpha}}|u^{\sigma}_{xxx}|^{\frac{1-\alpha}{\alpha}}|u^{\sigma}_{xxx}|^{2} \, dx \, dt \\
        & \quad + \int_{\{|u^{\sigma}_{xxx}| > \sigma\}} |u^{\sigma}|^{(\alpha+2)\frac{\alpha+1}{\alpha}} \bigl(|u^{\sigma}_{xxx}|^2 + \sigma^2\bigr)^{\frac{\alpha-1}{2}\frac{\alpha+1}{\alpha}}|u^{\sigma}_{xxx}|^{\frac{\alpha+1}{\alpha}} \, dx \, dt \\
         \leq & C T_{\sigma} \sigma^{\alpha+1} + C\int_0^{T_{\sigma}} D_\sigma[u^{\sigma}](t) \, dt\\
        \leq & C\bigl(T_{\sigma} \sigma^{\alpha+1} + E[u_0] \bigr).
    \end{align*}
  
    \noindent\textbf{(iii)} Since $u^\sigma$ is a weak solution to \eqref{eq:PDE_reg}, we have 
    \begin{equation*}
     \int_0^{T_{\sigma}} \langle u^\sigma_t, \varphi \rangle_{W^1_{\alpha+1}(\Omega)}\, dt
     =
     \int_0^{T_{\sigma}} \int_{\Omega} (u^\sigma)^{\alpha+2}\,  \psi_\sigma(u^\sigma_{xxx})\, \varphi_x\, dx\, dt
    \end{equation*}
    for all $\varphi \in L_{\alpha+1}\bigl((0,T_\sigma); 
    W^1_{\alpha +1,B}(\Omega)\bigr)$. Applying Hölder's inequality and  (i),  we obtain
    \begin{align*}
    &
    \left|\int_0^{T_\sigma} \langle u^\sigma_t, \varphi \rangle_{W^1_{\alpha+1}(\Omega)}\, dt
    \right|
    \leq
    \int_0^{T_\sigma} \int_{\Omega} |u^\sigma|^{\alpha+2} \left|\psi_\sigma(u^\sigma_{xxx})\right|\, |\varphi_x|\, dx\, dt
    \\
    &\quad 
    \leq
    \left(\int_0^{T_\sigma} \int_{\Omega} |u^\sigma|^{\alpha+2} |\varphi_x|^{\alpha+1}\, dx\, dt
    \right)^\frac{1}{\alpha+1} \cdot
    \\
    &\quad\quad \cdot
    \left(\int_0^{T_\sigma} \int_{\Omega} |u^\sigma|^{\alpha+2} \bigl(|u^\sigma_{xxx}|^2 + \sigma^2\bigr)^{\frac{\alpha-1}{2}\frac{\alpha+1}{\alpha}} |u^\sigma_{xxx}|^\frac{\alpha+1}{\alpha}\, dx\, dt
    \right)^\frac{\alpha}{\alpha+1}
    \\
    &\quad 
    \leq
    C \|\varphi\|_{L_{\alpha+1}((0,T_{\sigma});W^1_{\alpha+1}(\Omega))}
    \left(\int_0^{T_\sigma} \int_{\Omega} |u^\sigma|^{\alpha+2} \bigl(|u^\sigma_{xxx}|^2 + \sigma^2\bigr)^{\frac{\alpha-1}{2}\frac{\alpha+1}{\alpha}} |u^\sigma_{xxx}|^\frac{\alpha+1}{\alpha}\, dx\, dt
    \right)^\frac{\alpha}{\alpha+1}.
\end{align*}
For $0<\alpha<1$, we obtain similar as in step (ii) that 
\begin{equation*}
  \int_0^{T_\sigma} \int_{\Omega} |u^\sigma|^{\alpha+2} \bigl(|u^\sigma_{xxx}|^2 + \sigma^2\bigr)^{\frac{\alpha-1}{2}\frac{\alpha+1}{\alpha}} |u^\sigma_{xxx}|^\frac{\alpha+1}{\alpha}\, dx\, dt 
   \leq  \int_0^{T_\sigma} D_\sigma[u^\sigma](t) \, dt 
   \leq E[u_0]. 
\end{equation*}
For $1<\alpha < \infty$, we get, similarly as in step (ii), 
\begin{align*}
  \int_0^{T_\sigma} \int_{\Omega} |u^\sigma|^{\alpha+2} \bigl(|u^\sigma_{xxx}|^2 + \sigma^2\bigr)^{\frac{\alpha-1}{2}\frac{\alpha+1}{\alpha}} |u^\sigma_{xxx}|^\frac{\alpha+1}{\alpha}\, dx\, dt 
   &\leq  C T_\sigma \sigma^{\alpha+1} + C \int_0^{T_\sigma} D_\sigma[u^\sigma](t) \, dt 
  \\
  & \leq C\left(T_\sigma \sigma^{\alpha+1} + E[u_0]\right). 
\end{align*}

    \noindent\textbf{(iv)} We prove that $(u_{xxx}^\sigma)_\sigma$ is 
    uniformly bounded in $L_{\alpha+1,\mathrm{loc}}\bigl((0,T_\sigma)\times \Omega\bigr)$. Note that by definition of \(T_{\sigma}\) and continuity of \(u^{\sigma}\), we have \(u^{\sigma}(t,x) > c_\delta > 0\) for all \((t,x) \in [0,T_{\sigma}-\delta)\times \Omega\), for every \( \delta >0 \). 
    
    In the case $1<\alpha<\infty$, we get 
    \begin{align*}
    \int_0^{T_\sigma -\delta} \int_{\Omega} | u_{xxx}^\sigma|^{\alpha+1} \, dx\, dt 
    &\leq \int_0^{T_\sigma-\delta} \int_{\Omega}  \bigl(|u^\sigma_{xxx}|^2 + \sigma^2\bigr)^{\frac{\alpha-1}{2}} | u_{xxx}^\sigma|^2  \, dx\, dt
    \\
    &\leq C \int_0^{T_\sigma -\delta}D_\sigma[u^\sigma](t)\, dt
    \\
    &\leq C E[u_0], 
    \end{align*}
    where the constant $C$ depends also on $\delta$, and where in the last step we used \eqref{eq:energy-diss-sigma}. 
    
    Now we consider the case $0<\alpha<1$. We have 
    \begin{align*}
    \int_0^{T_\sigma-\delta} \int_{\Omega} | u_{xxx}^\sigma|^{\alpha+1} \, dx\, dt 
    & = \int_{\{|u^{\sigma}_{xxx}| \leq \sigma\}} 
    |u_{xxx}^\sigma|^{\alpha+1} \, dx\, dt  + 
    \int_{\{|u^{\sigma}_{xxx}| > \sigma\}} 
    |u_{xxx}^\sigma|^{\alpha+1} \, dx\, dt  
    \\
    &\leq C (T_\sigma-\delta) \sigma^{\alpha+1} + \int_{\{|u^{\sigma}_{xxx}| > \sigma\}} 
    |u_{xxx}^\sigma|^{\alpha+1} \, dx\, dt . 
    \end{align*}
    Using the inequality 
    \begin{align*}
    |x|^{\alpha+1} = \left( \tfrac{1}{2} |x|^2 + \tfrac{1}{2} |x|^2
    \right)^{\frac{\alpha-1}{2}} |x|^2 
    \leq \left(\tfrac{1}{2} \right)^{\frac{\alpha-1}{2}} 
    \left( |x|^2  + \sigma^2\right)^{\frac{\alpha-1}{2}}  |x|^2, 
    \qquad |x|>\sigma, \; x \in \R, 
    \end{align*}
    we obtain 
   \begin{align*}
    \int_0^{T_\sigma -\delta} \int_{\Omega} | u_{xxx}^\sigma|^{\alpha+1} \, dx\, dt 
    &\leq C (T_\sigma-\delta) \sigma^{\alpha+1} 
    + C  \int_0^{T_\sigma -\delta} \int_{\Omega} 
    \bigl(|u^\sigma_{xxx}|^2 + \sigma^2\bigr)^{\frac{\alpha-1}{2}} | u_{xxx}^\sigma|^2  \, dx\, dt
    \\
    &\leq C (T_\sigma-\delta)  \sigma^{\alpha+1}  
    + C \int_0^{T_{\sigma}-\delta} D_\sigma[u^\sigma](t) \, dt
    \\
    &\leq C \left((T_\sigma-\delta) \sigma^{\alpha+1}  
    +  E[u_0]\right) 
    \end{align*}
    with $C$ depending also on $\delta$. In the last step we used again \eqref{eq:energy-diss-sigma}. 
    
    \noindent\textbf{(v)} As observed in (i), \(u^{\sigma}\) is uniformly bounded in \(L_{\infty}\bigl((0,T_{\sigma});L_{\infty}(\Omega)\bigr)\), and hence also in \(L_{\alpha+1}\bigl((0,T_{\sigma})\times \Omega\bigr)\). From (iv), we also know that \(u_{xxx}^{\sigma}\) is uniformly bounded in \(L_{\alpha+1,\mathrm{loc}}\bigl((0,T_{\sigma})\times \Omega\bigr)\). Combining this, we find that \(u^{\sigma}\) is uniformly bounded in \(L_{\alpha+1,\mathrm{loc}}\bigl((0,T_{\sigma});W_{\alpha+1,B}^3(\Omega)\bigr)\) by interpolation.
    
    \noindent\textbf{(vi)} This follows as in (iii) using a duality argument.
\end{proof}

%-------------------------------------------------------
%-------------------------------------------------------
\bigskip

\subsection{Proof of Theorem {\ref{thm:Local_Ex_PL}}: Local Existence of Positive Weak Solutions to the Original Problem} \label{ssec:Existence_PL}

In this section we pass to the limit of a vanishing regularisation parameter $\sigma \searrow 0$. Using the uniform bounds provided in Lemma \ref{lem:uniform_bounds}, we show that the family $(u^\sigma)_\sigma$ admits an accumulation point that is a positive weak solution to the original problem \eqref{eq:PDE}. As usual, we use Minty's trick in order to identify the (nonlinear) limit flux.

%-----------------------------------------------------
%-----------------------------------------------------

\begin{lemma}[Convergence of approximations] \label{lem:convergence}
Let $u^\sigma$ be the maximal solution to \eqref{eq:PDE_reg} for a fixed $\sigma \in (0,1)$ and a positive initial value $u_0 \in W^{4\rho}_{p,B}(\Omega)$ such that $u_0(x) > 0\,$ for all $x \in \bar{\Omega}$.
Then the following holds true. There are a positive time $T > 0$ and a subsequence $(u^\sigma)_\sigma$ (not relabelled) such that, as $\sigma \searrow 0$, we have convergence in the following sense:\\[-0.3cm]
\begin{itemize}
	\item[(i)] $u^\sigma \to u$ strongly in $C\bigl([0,T];C^{\rho}(\bar{\Omega})\bigr)$;\\[-0.3cm] %\cap L_{\alpha+1}\bigl((0,T);C^{2+\sigma}(S^1)\bigr)$ for all $\rho \in [0,1/2)$ and $\sigma \in [0,\alpha/(\alpha+1))$;
	\item[(ii)] $\left|u^\sigma\right|^{\alpha+2} \psi_\sigma(u^\sigma_{xxx}) \rightharpoonup \chi$ weakly in $L_\frac{\alpha+1}{\alpha}\bigl((0,T)\times \Omega\bigr)$ for some limit function $\chi$;\\[-0.3cm]
	\item[(iii)] $u^\sigma_t \rightharpoonup u_t$ weakly in 	$L_\frac{\alpha+1}{\alpha}\bigl((0,T); (W^1_{\alpha+1,B}(\Omega))'\bigr)$;\\[-0.3cm]
	\item[(iv)] $u^\sigma_{xxx} \rightharpoonup u_{xxx}$ weakly in $L_{\alpha+1}\bigl((0,T)\times \Omega\bigr)$;\\[-0.3cm]
	\item[(v)] $(u^\sigma_x)_t \rightharpoonup u_{xt}$ weakly in $L_\frac{\alpha+1}{\alpha}\bigl((0,T);\bigl(W^1_{\alpha+1,0}(\Omega)\cap W^2_{\alpha+1}(\Omega)\bigr)'\bigr)$.
\end{itemize}
\end{lemma}

Since the proof of this lemma differs only very slightly from that in \cite{ansini_doubly_2004,lienstromberg_analysis_2022,lienstromberg_long-time_2022}, we shift it to the appendix.

%-------------------------------------------------------
%-------------------------------------------------------

We are left to prove the convergence of the nonlinear flux term
$\bigl(\left|u^\sigma\right|^{\alpha+2} \psi_\sigma(u^\sigma_{xxx}\bigr)\bigr) \rightharpoonup \bigl(\left|u\right|^{\alpha+2} \psi(u_{xxx})\bigr)$ in $L_\frac{\alpha+1}{\alpha}\bigl((0,T)\times \Omega\bigr)$. This is done in the next lemma the proof of which is based on the monotonicity of the regularisation and  Minty's trick.

%-------------------------------------------------------
%-------------------------------------------------------

\begin{lemma}\label{lem:limit_flux}
Given $\sigma \in (0,1)$, let $u^\sigma$ be the maximal solution to \eqref{eq:PDE_reg}, corresponding to an initial value $u_0 \in W^{4\rho}_{\alpha+1,B}(\Omega)$. Then there exists a subsequence $(u^\sigma)_\sigma$ (not relabelled) such that
\begin{equation*}
	\left|u^\sigma\right|^{\alpha+2} \psi_\sigma(u^\sigma_{xxx})\ 
	\rightharpoonup
	|u|^{\alpha+2} \psi(u_{xxx}) 
	\quad
	\text{weakly in }
	L_\frac{\alpha+1}{\alpha}\bigl((0,T)\times \Omega\bigr)
\end{equation*}
as $\sigma \searrow 0$.
\end{lemma}

The proof of the above stated lemma uses the same arguments as the one in \cite{lienstromberg_analysis_2022}. For the sake of completeness, we include it in the appendix. 

%-----------------------------------------------------
%-----------------------------------------------------

\begin{remark}
\label{Rem:cont_H1}
Note that the limit $u$ is bounded in $C\bigl([0,T];H^1(\Omega)\bigr)$. Indeed, from Lemma \ref{lem:convergence} (i) we already know that
\begin{equation*}
    u \in C\bigl([0,T];C^\rho(\bar{\Omega})\bigr) 
    \hookrightarrow
    C\bigl([0,T];L_2(\Omega)\bigr).
\end{equation*}
Furthermore,
\begin{equation*}
    u_x \in L_{\alpha+1}\bigl((0,T);W^1_{\alpha+1,0}(\Omega) \cap W^2_{\alpha+1}(\Omega)\bigr) \quad \text{and} \quad
	u_{xt}\in L_\frac{\alpha+1}{\alpha}\bigl((0,T);\bigl(W^1_{\alpha+1,0}(\Omega)\cap W^2_{\alpha+1}(\Omega)\bigr)'\bigr)
\end{equation*}
due to Lemma \ref{lem:convergence} (iv) and (v) and lower semicontinuity of the norm. Using \cite[Remark 3.4]{bernis_existence_1988}, this yields that $u_x \in C\bigl([0,T];L_2(\Omega)\bigr)$. Therefore, $u \in C\bigl([0,T];H^1(\Omega)\bigr)$.
\end{remark}

\begin{proof}[\textbf{Proof of Theorem \ref{thm:Local_Ex_PL}}]
 \noindent\textbf{(i)} We first show that the limit $u$ is 
bounded away from zero on $[0,T] \times \bar{\Omega}$. 
This follows immediately from the positivity of $u^\sigma$ on 
$[0,T_\sigma)\times \bar{\Omega}$ and the convergence in 
Lemma \eqref{lem:convergence} (i). 

 \noindent\textbf{(ii)} Thanks to Lemma \ref{lem:uniform_bounds} (iii) and (iv) and Remark 
 \ref{Rem:cont_H1} above, we obtain the regularity properties 
 \[
 u \in C\bigl([0,T]; H^1(\Omega)\bigr) \cap L_{\alpha+1}\bigl((0,T); W^3_{\alpha+1,B}(\Omega)\bigr)
 \quad \text{and} \quad u_t \in L_{\frac{\alpha+1}{\alpha}}\bigl(
 (0,T); (W^1_{\alpha+1,B}(\Omega))'\bigr). 
 \]
 
  \noindent\textbf{(iii)} We now prove that $u$ satisfies the weak integral formulation in Definition \ref{def:weak_sol}. 
  To do so, note that for solutions to the regularised problem \eqref{eq:PDE_reg} we have that 
  \[
  \int_0^T \langle u_t^\sigma ,\phi\rangle_{W^1_{\alpha+1}(\Omega)}\, dt
        =
        \int_0^T \int_\Omega |u^\sigma|^{\alpha+2} \psi_{\sigma}(u^\sigma_{xxx})\, \phi_x\, dx\, dt
  \]
 for all test functions $\phi \in L_{\alpha+1}\bigl((0,T); W^1_{\alpha+1,B}(\Omega)\bigr)$. On the one hand, since $\phi_x \in L_{\alpha+1}\bigl((0,T)\times \Omega)\bigr)$, it follows from Lemma 
 \ref{lem:limit_flux} that 
   \[
  \int_0^T \langle u_t^\sigma ,\phi\rangle_{W^1_{\alpha+1}(\Omega)}\, dt
        \longrightarrow
        \int_0^T \int_\Omega |u|^{\alpha+2} \psi(u_{xxx})\, \phi_x\, dx\, dt.
  \]
 On the other hand, Lemma \ref{lem:uniform_bounds} (iii) gives 
\[
  \int_0^T \langle u_t^\sigma ,\phi\rangle_{W^1_{\alpha+1}(\Omega)}\, dt
        \longrightarrow
        \int_0^T \langle u_t ,\phi\rangle_{W^1_{\alpha+1}(\Omega)}\, dt. 
\]
Combining both, we then find that $u$ satisfies the desired integral identity
 \[
  \int_0^T \langle u_t ,\phi\rangle_{W^1_{\alpha+1}(\Omega)}\, dt
        =
        \int_0^T \int_\Omega u^{\alpha+2} \psi(u^\sigma_{xxx})\, \phi_x\, dx\, dt
  \]
 for all $\phi \in L_{\alpha+1}\bigl((0,T); W^1_{\alpha+1,B}(\Omega)\bigr)$. 
 
  \noindent\textbf{(iv)} By Lemma \ref{lem:convergence} (i) the initial condition is satisfied in the limit. That the first boundary condition in $\eqref{eq:PDE}_2$ is fulfilled by $u$ follows from Lemma \ref{lem:convergence} (v). 
  
 \noindent\textbf{(v)} This follows from the conservation of mass property 
 \[
 \int_{\Omega} u^\sigma (t)\, dx = \int_{\Omega} u_0\, dx,
\quad t\in [0,T_\sigma),
 \]
 for the approximation $u^\sigma$ (see Theorem \ref{thm:Local_Ex_PL_reg} (ii)) and the convergence in Lemma \ref{lem:convergence} (i). 
 
 \noindent\textbf{(vi)} In Lemma \ref{lem:limit_flux} we have already shown that the solution $u$ to the original problem 
 \eqref{eq:PDE} satisfies the energy-dissipation identity for almost every $t\in [0,T]$. 
\end{proof}

%-----------------------------------------------------
%-----------------------------------------------------
\bigskip

\section{Differential Inequality for the Energy and Regularity Estimates}\label{sec:regularity-estimates}

The content of this section is twofold. First, we derive a differential inequality of {\L}ojasiewicz--Simon type for the energy functional $E$ which is valid as long as the weak solution to \eqref{eq:PDE} remains bounded away from zero. Then, we derive $L_1$-in-time regularity estimates for the weak solution to \eqref{eq:PDE}. The results are the same as in the cylindrical Taylor--Couette setting in \cite{lienstromberg_analysis_2022,lienstromberg_long-time_2022}. However, since the present paper deals with the flat case, the proofs cannot rely on Fourier analysis. 

\begin{proposition}\label{prop:ED_estimate}
 Fix $\alpha > 0$ and a positive initial value $u_0\in H^1(\Omega)$ with \(u_0(x) > 0\) for \(x\in \bar{\Omega}\). Let $u \in C\bigl([0,T];H^1(\Omega)\bigr) 
    \cap
    L_{\alpha+1}\bigl((0,T);W^3_{\alpha+1,B}(\Omega)\bigr)$
    with 
    $u_t \in L_\frac{\alpha+1}{\alpha}\bigl((0,T);(W^1_{\alpha+1,B}(\Omega))'\bigr)$
    be a weak solution to \eqref{eq:PDE} with initial value $u_0$, as obtained in Theorem \ref{thm:Local_Ex_PL}. Let \(m = \min_{(t,x)\in [0,T] \times\bar{\Omega}}u(t,x) > 0\).  Then there is a constant \(C = C_{\alpha,\Omega,m}>0\) such that
\begin{equation}
    \frac{d}{dt} E[u](t) = - D[u](t) \leq -C \bigl(E[u](t)\bigr)^{\frac{\alpha+1}{2}}
\end{equation}
for almost every $t\in [0,T]$.
\end{proposition}

%-----------------------------------------------------
%-----------------------------------------------------

The proof of Proposition \ref{prop:ED_estimate} is based on the following crucial Poincaré estimate. It is worthwhile to emphasise that this estimate is valid in both the shear-thinning case and the shear-thickening case.
%-----------------------------------------------------
%-----------------------------------------------------

\begin{lemma}\label{Lem:Poincare}
    Fix \(\alpha>0\) and let \(v\in H^1(\Omega) \cap W^3_{\alpha+1,B}(\Omega)\) with \(\bar{v} = 0\) and \(v_x(x) = 0\) for \(x\in \partial\Omega\). Then there exists a constant \(C=C_{\alpha,\Omega}>0 \) such that
        \[E[v] \leq C \|v_{xxx}\|_{L_{\alpha+1}(\Omega)}^2.\]
\end{lemma}

%-----------------------------------------------------
%-----------------------------------------------------

\begin{proof}
We distinguish the cases $\alpha = 1, \alpha > 1$ and $\alpha < 1$.

\noindent\textbf{The case $\alpha=1$.} This is just a direct application of Poincaré's inequality.
    
\noindent\textbf{The case $\alpha>1$.} Define \(w = v_{xx} \in W^1_{\alpha+1}(\Omega) \subset L_2(\Omega)\). Observe that \(v\) is a weak solution to the Neumann boundary-value problem given by
    \begin{equation*}
    \begin{cases}
        v_{xx}
        =
        w,
        &
        x\in \Omega,
        \\
        v_x 
        =
        0,
        &
        x\in \partial\Omega.
    \end{cases}
    \end{equation*}
    Hence, we obtain the estimate \( \|v_x\|_{L_2(\Omega)} \leq C \|w\|_{L_2(\Omega)} \). Furthermore, note that \(\bar{w}=0\). Using this, applying Poincaré's inequality and then Jensen's inequality for the concave function \( s \mapsto s^{2/(\alpha+1)}\), \( s\in 
    (0,\infty) \), we find that
    \begin{align*}
        E[v] = \frac{1}{2} \|v_x\|_{L_2(\Omega)}^2 \leq C\|w\|_{L_2(\Omega)}^2 \leq C\|w_x\|_{L_2(\Omega)}^2 = C \int_{\Omega} |v_{xxx}|^{(\alpha+1)\frac{2}{\alpha+1}} \, dx \leq C\left(\int_{\Omega} |v_{xxx}|^{\alpha+1}\, dx\right)^{\frac{2}{\alpha+1}}.
    \end{align*}
\noindent\textbf{The case $\alpha<1$.} In this case we have \(2/(\alpha+1) > 1\) and we cannot use Jensen's inequality anymore. Instead, we rely on the Sobolev embedding and a-priori estimates for the Bi-Laplace equation. Define \(w = v_{xxx}\in L_{\alpha+1}(\Omega)\). Then \(v\) is a weak solution to
    \begin{equation*}
    \begin{cases}
        v_{xxxx}
        =
        w_x,
        &
        x\in \Omega,
        \\
        v_x = v_{xxx}
        =
        0,
        &
        x\in \partial\Omega,
    \end{cases}
    \end{equation*}
    in the sense that
        \[\int_{\Omega} v_{xx} \varphi_{xx} \, dx = - \int_{\Omega} w \varphi_x \, dx \quad \text{ for all } \varphi \in W^2_{\frac{\alpha+1}{\alpha},B}(\Omega).\]
    Since \(v \in C^2(\bar{\Omega})\) by the Sobolev embedding, we may use \(v \in W^2_{\frac{\alpha+1}{\alpha},B}(\Omega)\) as a test function and find that
        \[\|v_{xx}\|_{L_2(\Omega)}^2 \leq \int_{\Omega} |w| |v_x| \, dx \leq \|w\|_{L_{\alpha+1}(\Omega)}\|v_x\|_{L_{\frac{\alpha}{\alpha+1}}(\Omega)} \leq C \|w\|_{L_{\alpha+1}(\Omega)}\|v_{xx}\|_{L_{2}(\Omega)}.\]
    Dividing by $\|v_{xx}\|_{L_{2}(\Omega)}$, we conclude that \(\|v_{xx}\|_{L_2(\Omega)}^2 \leq C\|w\|_{L_{\alpha+1}(\Omega)}^2\). Finally, the  desired estimate 
    \begin{align*}
        E[v] = \frac{1}{2} \|v_x\|_{L_2(\Omega)}^2 \leq C\|v_{xx}\|_{L_2(\Omega)}^2 \leq C\|w\|_{L_{\alpha+1}(\Omega)}^2 = C\|v_{xxx}\|_{L_{\alpha+1}(\Omega)}^2
    \end{align*}
    follows by Poincaré's inequality. 
\end{proof}

%-----------------------------------------------------
%-----------------------------------------------------

\begin{proof}[\textbf{Proof of Proposition \ref{prop:ED_estimate}}]
    From Theorem \ref{thm:Local_Ex_PL} we know that weak solutions to \eqref{eq:PDE} satisfy the energy-dissipation identity \eqref{eq:energy-dissipation}. Taking the derivative in time, we find that
    \begin{equation*}
        \frac{d}{dt} E[u](t) + D[u](t) = 0
    \end{equation*}
    for almost every $t\in [0,T]$.
    Furthermore, since \(m = \min_{(t,x)\in [0,T] \times\bar{\Omega}}u(t,x) > 0\) and by Lemma \ref{Lem:Poincare}, we obtain
    \begin{equation*}
        D[u](t) = \int_{\Omega} |u|^{\alpha+2}|u_{xxx}|^{\alpha+1} \, dx \geq 
        m^{\alpha+2}\|u_{xxx}(t)\|_{L_{\alpha+1}(\Omega)}^{\alpha+1} \geq Cm^{\alpha+2}\bigl(E[u](t)\bigr)^{\frac{\alpha+1}{2}}
    \end{equation*}
    for almost every \(t\in [0,T]\). This concludes the proof.
\end{proof}

%=============================================================================
%=============================================================================

Next, we turn to \(L_1\)-in-time bounds for the dissipation functional in terms of the energy. The proof is a simplified version of the one in \cite{lienstromberg_long-time_2022} for general degenerate parabolic problems of fourth order. In our case, it relies on testing the partial differential equation with a time cut-off of the second spatial derivative. 

\begin{theorem}\label{thm:L1-in-time}
Fix $\alpha > 0$ and a positive initial value $u_0\in H^1(\Omega)$ with \(u_0(x) > 0\) for \(x\in \bar{\Omega}\).
Let $u \in C\bigl([0,T];H^1(\Omega)\bigr) 
    \cap
    L_{\alpha+1}\bigl((0,T);W^3_{\alpha+1,B}(\Omega)\bigr)$
    with 
    $u_t \in L_\frac{\alpha+1}{\alpha}\bigl((0,T);(W^1_{\alpha+1,B}(\Omega))'\bigr)$
    be a positive weak solution to \eqref{eq:PDE} on $(0,T)$, as obtained in Theorem \ref{thm:Local_Ex_PL}. Then there exists a constant \(C>0\), independent of \(t\), such that the dissipation functional $D[u]$ enjoys the $L_1$-in-time bound
    \begin{equation*}
        \int_{t/2}^{t} D[u](s)\, ds \leq \frac{C}{t} \int_{t/4}^{t/2} E[u](s) \, ds \leq \frac{C}{4} E[u]\bigl(\tfrac{t}{4}\bigr).
    \end{equation*}
\end{theorem}

\begin{proof}
    We choose a cut-off function \(\chi\in C^{\infty}(\R)\) in time such that \(0\leq \chi \leq 1\), \(\chi(s) = 1\) for \(s\geq t/2\), \(\chi(s) = 0\) for \(s\leq t/4\) and \(\chi'(s) \leq C/t\) for some constant \(C\), independent of \(t\). Now we define the test function \(\varphi(s,x) = \chi(s) u_{xx}(s,x)\in L_{\alpha+1}\bigl((0,T);W^1_{\alpha+1,B}(\Omega)\bigr)\). Since \(u\) is a weak solution to \eqref{eq:PDE} on the time interval \([0,t]\), we obtain
    \begin{equation}\label{eq:L_1_aux_1}
        \int_{0}^t \langle u_t, \chi(s)u_{xx}\rangle_{W^1_{\alpha+1}} \, ds = \int_{0}^{t} \int_{\Omega} u^{\alpha+2}\psi(u_{xxx})u_{xxx} \chi(s) \, dx\,ds = \int_{0}^{t} \chi(s)D[u](s) \, ds.
    \end{equation}
    Moreover, since \(\chi(0) = 0\) and \(\chi(t) =1\), we have the inequality
    \begin{equation} \label{eq:L_1_aux_2}
        0 \leq E[u](t) 
        = 
        \int_{0}^{t} \frac{d}{ds} \bigl( \chi(s) E[u](s) \bigr) \, ds 
        =  
        \int_{0}^{t} \chi'(s) E[u](s) \, ds - \int_{0}^{t} \chi(s) \langle u_{s}, u_{xx} \rangle_{W^1_{\alpha+1}} \, ds.
    \end{equation}
    Combining \eqref{eq:L_1_aux_1} and \eqref{eq:L_1_aux_2} and using that \(\chi \equiv 1\) on \([t/2,t]\) and \(D[u](s)\geq 0\) for all \(0\leq s \leq t\), we conclude that
    \begin{align*}
        \int_{t/2}^{t} D[u](s) \, ds & \leq \int_{0}^{t} \chi(s)D[u](s) \, ds 
        \leq \int_{0}^{t} \chi'(s) E[u](s) \, ds 
        \leq \frac{C}{t}\int_{t/4}^{t/2} E[u](s) \, ds.
    \end{align*}
    Finally, since \(E[u]\) decreases along solutions, we may estimate
    \begin{equation*}
        \frac{C}{t}\int_{t/4}^{t/2} E[u](s) \, ds \leq \frac{C}{4}E[u]\bigl(\tfrac{t}{4}\bigr).
    \end{equation*}
    This completes the proof.
\end{proof}

%=============================================================================
%=============================================================================
%=============================================================================
\bigskip

\section{Shear-Thickening Power-Law Fluids ($\alpha < 1$) -- Global Existence and Convergence to Steady States in Finite Time} \label{sec:shear-thickening}

This section deals with the long-time asymptotics of shear-thickening power-law fluids, i.e. we consider flow-behaviour exponents $\alpha<1$ in \eqref{eq:PDE}. We prove that for positive initial values $u_0 \in H^1(\Omega)$ that are close to a steady state in the sense that
% Now we consider a positive initial value \(u_0\in H^1(\Omega)\) such that
\begin{equation*}
    \frac{1}{2}\bar{u}_0 < u_0(x) < 2 \bar{u}_0, 
    \quad x\in \bar{\Omega}, \qquad \text{where} \quad \bar{u}_0 = \fint_\Omega u_0\, dx,  
\end{equation*}
problem \eqref{eq:PDE} with $\alpha<1$ possesses a globally-in-time defined positive weak solution that converges to a steady state in finite time. As in the circular Taylor--Couette setting \cite{lienstromberg_analysis_2022}, the corresponding proof relies mainly on the differential inequality derived in Proposition \ref{prop:ED_estimate}. This differential inequality guarantees that the energy becomes zero in finite time $0 < t^\ast < \infty$. We construct a globally-in-time defined positive weak solution by constant extension at time $t^\ast$.

By Theorem \ref{thm:Local_Ex_PL} and Remark \ref{Rem:H1-initial} there exists a weak solution $u \in C\bigl([0,T];H^1(\Omega)\bigr) 
    \cap
    L_{\alpha+1}\bigl((0,T);W^3_{\alpha+1,B}(\Omega)\bigr)$ with $u_t \in L_\frac{\alpha+1}{\alpha}\bigl((0,T);(W^1_{\alpha+1,B}(\Omega))'\bigr)$ to \eqref{eq:PDE}. 
We define the time
\begin{equation}
\label{eq:time_tau}
    \tau
    =
    \sup\{\tilde{T} > 0;\ \text{$\exists$ a weak solution $u$ to \eqref{eq:PDE} on } [0,\tilde{T}]
    \text{ with } \tfrac{1}{2} \bar{u}_0 \leq u(t,x) \leq 2 \bar{u}_0\ \forall\ 0 \leq t \leq \tilde{T}\},
\end{equation}
up to which solutions are bounded away from zero and bounded above. Note that by continuity of weak solutions, we have \(0<\tau\). By Remark \ref{Rem:extension} we may also assume that \(\tau \leq T\). In particular, we can apply the results of Section \ref{sec:regularity-estimates} up to time \(\tau\).

%-----------------------------------------------------
%-----------------------------------------------------

\begin{theorem}[Global existence and convergence in finite time]\label{Thm:Shear_thick}
    Fix \(0<\alpha < 1\). There exists \(\eps >0\) such that, for all positive initial values \(u_0\in H^1(\Omega)\) with \(\|u_0-\bar{u}_0\|_{H^1(\Omega)} < \eps\), there is a positive global weak solution 
    \begin{equation*}
        u\in C\bigl([0,\infty);H^1(\Omega)\bigr) \cap L_{\alpha+1,\mathrm{loc}}\bigl((0,\infty);W^3_{\alpha+1,B}(\Omega)\bigr) \quad  \text{with} \quad
    u_t \in L_{\frac{\alpha+1}{\alpha},\text{loc}}\bigl((0,\infty);(W^1_{\alpha+1,B}(\Omega))'\bigr).
    \end{equation*}
    Moreover, there exists a time \(0<t^* < \infty\) such that
         \[u(t,\cdot) \longrightarrow \bar{u}_0 \text{ in } H^1(\Omega), \text{ as } t\to t^*, \quad \text{and} \quad  u(t,x) = \bar{u}_0, \quad t \geq t^*,\ x\in \bar{\Omega}.
        \]
\end{theorem}

%-----------------------------------------------------
%-----------------------------------------------------

\begin{proof}
    Let \(u \in C\bigl([0,T];H^1(\Omega)\bigr) \cap L_{\alpha+1}\bigl((0,T);W^3_{\alpha+1,B}(\Omega)\bigr) \) the solution to \eqref{eq:PDE} provided by Theorem \ref{thm:Local_Ex_PL} and Remark \ref{Rem:H1-initial} with initial datum \(u_0 > \bar{u}_0/2\) in \( \bar{\Omega}\). Write \(u(t,x) = \bar{u}_0 + v(t,x)\) for \( (t,x) \in [0,T]\times \Omega \), where due to conservation of mass \(\int_{\Omega} v \, dx= 0\) for all \(t\in [0,T]\times \bar{\Omega}\). Then, by continuity and the definition of \(\tau\), we have \(|v(t,x)| \leq \bar{u}_0/2\) for \((t,x)\in [0,\tau]\times\bar{\Omega}\). Thus, there exists a constant \(C>0\) such that for almost every \(t\in [0,\tau]\) it holds
    	\[\int_{\Omega} |v_{xxx}|^{\alpha+1} \, dx 
    	\leq 
    	C \int_{\Omega} |u|^{\alpha+2} |v_{xxx}|^{\alpha+1}\, dx.\]
    Hence, using the energy-dissipation identity \eqref{eq:energy-dissipation} and Lemma \ref{Lem:Poincare}, we obtain 
    	\[\frac{d}{dt} E[v](t) = \frac{d}{dt} E[u](t) 
    	= 
    	- \int_{\Omega}|u|^{\alpha+2} |v_{xxx}|^{\alpha+1}\, dx 
    	\leq 
    	-C \|v_{xxx}(t)\|_{L_{\alpha+1}(\Omega)}^{\alpha+1} \leq -C \bigl(E[v](t)\bigr)^{\frac{\alpha+1}{2}}\]
    for almost every \(t\in [0,\tau]\). This inequality implies that the energy \(E[v](\cdot)=E[u](\cdot)\) is decreasing and hence \(\tau = T\). Furthermore, it follows that
    	\[\frac{d}{dt} \left( \bigl(E[v](t)\bigr)^{\frac{1-\alpha}{2}} \right) \leq -C_{\alpha},\]
    as long as \(E[v](t) >0\), and integration from \(0\) to \(t\) yields
    	\[\bigl(E[v](t)\bigr)^{\frac{1-\alpha}{2}} 
    	\leq 
    	\bigl(E[v_0]\bigr)^{\frac{1-\alpha}{2}} - C_{\alpha}t, \quad t\in [0,T], \text{ if } E[v](t) >0.\]
   	Thus, we conclude that
   		\[E[v](t) 
   		\leq 
   		\left(\bigl(E[v_0]\bigr)^{\frac{1-\alpha}{2}} - C_{\alpha}t\right)^{\frac{2}{1-\alpha}}, \quad t\in [0,T], \text{ if } E[v](t) >0,\]
    which implies the existence of a finite time \(t^* \geq 0\) with \(t^* \leq (E[v_0])^{\frac{1-\alpha}{2}}/C_{\alpha}\) such that
    	\[E[v](t) = 0, \quad t \geq t^*.\]
    We may choose \(\eps >0\) small enough so that we obtain \(t^* < T\). Finally, note that \(E[v](t) = 0\) for \(t \geq t^*\) and \(\bar{v}(t)=0\) implies that \(v(t,x)=0\) for all \(t\geq t^*\) and \(x\in \Omega\). Hence, the solution \(u\) may be extended by the constant solution \(\bar{u}_0\) for times \(t\geq t^*\) to a global-in-time weak solution \(u\in C\bigl([0,\infty);H^1(\Omega)\bigr) \cap L_{\alpha+1,\mathrm{loc}}\bigl((0,\infty);W^3_{\alpha+1,B}(\Omega)\bigr)\) and we have
    	\[u(t,x) \longrightarrow \bar{u}_0 \quad \text{in } H^1(\Omega)\]
     and uniformly as \(t\to t^*\) in finite time.
\end{proof}

%=============================================================================
%=============================================================================
%=============================================================================

\section{Shear-Thinning Power-Law Fluids ($\alpha > 1$) -- Global Existence and Polynomial Stability of Steady States} \label{sec:shear-thinning}

In this section we study the long-time behaviour of solutions to the  shear-thinning power-law equation. More precisely, we fix a flow-behaviour exponent $\alpha>1$ in \eqref{eq:PDE} and consider positive initial values $u_0 \in H^1(\Omega)$ that are close to a steady state in the sense that
\begin{equation*}
    \frac{1}{2}\bar{u}_0 < u_0(x) < 2 \bar{u}_0, 
    \quad x\in \bar{\Omega}, 
\end{equation*}
where $\bar{u}_0 = \fint_\Omega u_0\, dx$.   
We show that there exist global positive weak solutions $u$ to \eqref{eq:PDE} with $\alpha>1$ that remain $\eps$-close to the steady state for all times and converge at rate $1/t^\frac{1}{\alpha-1}$ to equilibrium, as $t\to \infty$.
Note that convergence to equilibrium has already been proved for the global non-negative weak solutions constructed in \cite{ansini_doubly_2004}, but with no rate of convergence.
The result on the rate of convergence is the same as in \cite{lienstromberg_long-time_2022} for the cylindrical Taylor--Couette setting. The proof relies again on the differential inequality for the energy, derived in Proposition \ref{prop:ED_estimate}. However, in the shear-thinning case also the $L_1$-in-time bound of Theorem \ref{thm:L1-in-time} is crucial.

\begin{theorem}[Global existence and polynomial stability]\label{Thm:Shear_thin}
    Fix \(1<\alpha < \infty\). There exists \(\eps >0\) such that for all positive initial values \(u_0\in H^1(\Omega)\) with \(\|u_0-\bar{u}_0\|_{H^1(\Omega)} < \eps\), there is a global positive weak solution
    \begin{equation*}
        u\in C\bigl([0,\infty);H^1(\Omega)\bigr) \cap L_{\alpha+1,\mathrm{loc}}\bigl((0,\infty);W^3_{\alpha+1,B}(\Omega)\bigr) \quad  \text{with} \quad
    u_t \in L_{\frac{\alpha+1}{\alpha},\text{loc}}\bigl((0,\infty);(W^1_{\alpha+1,B}(\Omega))'\bigr).
    \end{equation*}
    Moreover, there is a constant \( C>0 \) such that
    \begin{equation*}
        \|u(t)-\bar{u}_0\|_{H^1(\Omega)} \leq \frac{C\eps}{\bigl(1+C\eps^{\alpha-1}t\bigr)^{\frac{1}{\alpha-1}}},
        \quad
        0 \leq t < \infty.
    \end{equation*}
    Furthermore, the dissipation decreases polynomially along the solution in the following \(L_1\)-in-time sense
    \begin{equation}\label{eq:L1-in-time-dissipation-thin}
        \int_{t/2}^{t} D[u](s) \, ds \leq \frac{C\eps^2}{\bigl(1+C\eps^{\alpha-1}t\bigr)^{\frac{2}{\alpha-1}}}
    \end{equation}
    for all $0 \leq t < \infty$.
\end{theorem}

\begin{remark}
    Note that the weak solution \(u\in C\bigl([0,\infty);H^1(\Omega)\bigr) \cap L_{\alpha+1,\mathrm{loc}}\bigl((0,\infty);W^3_{\alpha+1,B}(\Omega)\bigr)\) obtained in Theorem \ref{Thm:Shear_thin} satisfies \(u(t,x) \geq \bar{u}_0/2\) for all \((t,x)\in [0,\infty)\times\bar{\Omega}\). Hence, \eqref{eq:L1-in-time-dissipation-thin} implies that the \(W^3_{\alpha+1}(\Omega)\)-norm is also controlled in the \(L_1\)-in-time sense by
    \begin{equation*}
        \int_{t/2}^{t} \int_{\Omega} |u_{xxx}(s)|^{\alpha+1} \, dx\,ds \leq  \frac{C\eps}{\bigl(1+C\eps^{\alpha-1}t\bigr)^{\frac{1}{\alpha-1}}}
    \end{equation*}
    for all $0 \leq t < \infty$.
\end{remark}

%-----------------------------------------------------
%-----------------------------------------------------

\begin{proof}[\textbf{Proof of Theorem  \ref{Thm:Shear_thin}}]
	First, we show that there exists an \(\eps> 0\) such that for all initial values \(u_0 \in H^1(\Omega)\) with \(\bar{u}_0=0\) and \(\|u_0- \bar{u}_0\|<\eps\), there is a constant \(C>0\) independent of \(\eps\) such that
	\begin{equation*}
	    E[u](t) \leq \frac{C\eps^2}{\bigl(1+\eps^{\alpha-1}t\bigr)^{\frac{2}{\alpha-1}}},
	    \quad 0 \leq t < \infty.
	\end{equation*}
	 Let \(u \in C\bigl([0,T];H^1(\Omega)\bigr) \cap L_{\alpha+1}\bigl((0,T);W^3_{\alpha+1,B}(\Omega)\bigr) \) the solution to \eqref{eq:PDE} provided by Theorem \ref{thm:Local_Ex_PL} and Remark \ref{Rem:H1-initial} with initial datum $u_0 \in H^1(\Omega)$ satisfying \(u_0 > \bar{u}_0/2\) in \( \bar{\Omega}\). As in the proof of Theorem \ref{Thm:Shear_thick}, we  write \(u(t,x) = \bar{u}_0 + v(t,x)\) for \( (t,x) \in [0,T] \times \bar{\Omega}\), where due to conservation of mass \(\int_{\Omega} v\, dx= 0\) for all \(t\in [0,T]\). Then, by continuity and the definition of \(\tau\) (see \eqref{eq:time_tau}), we have \(|v(t,x)| \leq \bar{u}_0/2\) for \((t,x) \in [0,\tau]\times \bar{\Omega}\). By Lemma \ref{Lem:Poincare} we then conclude that
	 	\[ E[u](t)=E[v](t) 
	 	\leq 
	 	C \left(\int_{\Omega} |v_{xxx}|^{\alpha+1}\,dx\right)^{\frac{2}{\alpha+1}} 
	 	\leq 
	 	C \left(\int_{\Omega} |u|^{\alpha+2}|v_{xxx}|^{\alpha+1}\, dx\right)^{\frac{2}{\alpha+1}} 
	 	= C \bigl(D[u](t)\bigr)^{\frac{2}{\alpha+1}}\]
	 for almost every \(t \in [0,\tau]\). Inserting this into the energy-dissipation identity \eqref{eq:energy-dissipation}, we find that
	 	\begin{equation}
	 	    \label{eq:Energy_dec_Shear_thin}
	 	\frac{d}{dt} E[u](t) = - D[u](t) \leq -C \bigl(E[u](t)\bigr)^{\frac{\alpha+1}{2}}
	 	\end{equation}
	 for almost every \(t \in [0,\tau]\). This implies that the energy \(E[u](\cdot)\) is decreasing and hence \(\tau = T\). Furthermore,  we can rewrite estimate \eqref{eq:Energy_dec_Shear_thin} as
	 	\[\frac{2}{1-\alpha} \frac{d}{dt} \bigl( E[u](t) \bigr)^{\frac{1-\alpha}{2}} \leq - C,
	 	\quad 0 \leq t \leq \tau,\]
	 so that, after integration, we obtain
		\[\frac{2}{1-\alpha} \bigl(E[u](t)\bigr)^{\frac{1-\alpha}{2}} \leq -Ct + \frac{2}{1-\alpha} \bigl(E[u_0]\bigr)^{\frac{1-\alpha}{2}}, \quad 0\leq  t \leq \tau.\]
	Since \(\alpha >1\), we can rearrange this inequality to
		\[E[u](t) \leq \left(\bigl(E[u_0]\bigr)^{\frac{1-\alpha}{2}} + \frac{C(\alpha-1)}{2}t \right)^{\frac{2}{1-\alpha}} = \frac{E[u_0]}{\left(1+C\bigl(E[u_0]\bigr)^{\frac{\alpha-1}{2}}t\right)^{\frac{2}{\alpha-1}}}, \quad 0 \leq t \leq \tau.\]
	Since the function \(s\mapsto \frac{s}{\bigl(1+Cs^{\frac{\alpha-1}{2}}t\bigr)^{\frac{2}{\alpha-1}}}\) is increasing on $[0,\infty)$ and \(E[u_0] \leq \eps^2\) by assumption, we infer that
	    \[E[u](t) \leq \frac{C\eps^2}{\bigl(1+C\eps^{\alpha-1}t\bigr)^{\frac{2}{\alpha-1}}}, \quad 0 \leq t \leq \tau.\]
	Now, we choose \(\eps >0\) such that
    \begin{equation*}
        \|u_0-\bar{u}_0\|_{L_{\infty}(\Omega)} \leq C \bigl(E[u_0]\bigr)^{\frac{1}{2}} \leq \frac{\bar{u}_0}{2},
    \end{equation*}
    where the first estimate is due to the embedding $H^1(\Omega) \hookrightarrow L_{\infty}(\Omega)$ and the Poincaré inequality.
    This, together with the fact that \( E[u](\cdot)\) is decreasing, guarantees that
    \begin{equation*}
        \|u(t)- \bar{u}_0\|_{L_{\infty}(\Omega)} 
        \leq 
        C \bigl(E[u](t)\bigr)^{\frac{1}{2}}  
        \leq 
        C \bigl(E[u_0]\bigr)^{\frac{1}{2}}  
        \leq \frac{\bar{u}_0}{2},
        \quad 0\leq t \leq \tau.
    \end{equation*}
    Hence, solutions $u$ to \eqref{eq:PDE} on $[0,\tau]$ remain strictly bounded away from zero and by bootstrapping as in Remark \ref{Rem:extension}, we may extend it beyond time \(\tau\) to a global-in-time weak solution \(u\in C\bigl([0,\infty);H^1(\Omega)\bigr)\) \\ \(\cap L_{\alpha+1,\mathrm{loc}}\bigl((0,\infty);W^3_{\alpha+1,B}(\Omega)\bigr)\)  that satisfies 
    \begin{equation*}
        E[u](t) \leq \frac{C\eps^2}{\bigl(1+C\eps^{\alpha-1}t\bigr)^{\frac{2}{\alpha-1}}}, \quad 0 \leq  t <\infty.
    \end{equation*}
    Since by Poincaré's inequality we have
    \begin{equation*}
        \|u(t)-\bar{u}_0\|_{H^1(\Omega)} \leq C \sqrt{E[u](t)} \leq \frac{C\eps}{\bigl(1+C\eps^{\alpha-1}t\bigr)^{\frac{1}{\alpha-1}}},
        \quad 0 \leq  t <\infty,
    \end{equation*}
    we conclude the polynomial stability in \(H^1(\Omega)\). For the \(L_1\)-in-time estimate, we apply Theorem \ref{thm:L1-in-time} and obtain
    \begin{equation*}
        \int_{t/2}^{t} D[u](s) \, ds \leq C E[u]\bigl(\tfrac{t}{4}\bigr) \leq \frac{C\eps^2}{\bigl(1+C\eps^{\alpha-1}t\bigr)^{\frac{2}{\alpha-1}}},
        \quad 0 \leq  t <\infty.
    \end{equation*}
    This completes the proof.
\end{proof}

%-----------------------------------------------------
%-----------------------------------------------------
\bigskip

\section{Global Existence and Exponential Stability for the Ellis-Law Thin-Film Equations} \label{sec:Ellis}

Now we turn to fluids with Ellis-law rheology. These are fluids whose viscosity approaches a Newtonian plateau for low shear rates, while for big shear rates the viscosity is shear-thinning. The corresponding thin-film equation is given by
\begin{equation} \label{eq:PDE-Ellis}
    \begin{cases}
        u_t + \bigl(u^{3}(1+ |uu_{xxx}|^{\alpha-1}) u_{xxx}\bigr)_x
        =
        0,
        & t>0,\ x \in \Omega,
        \\
        u_x(t,x) = u_{xxx}(t,x) = 0, 
        &
        t > 0,\ x \in \partial\Omega,
        \\
        u(0,x) = u_0(x),
        &
        x \in \Omega,
    \end{cases}
\end{equation}
for flow-behaviour exponents $\alpha \geq 1$. Here \(\Omega\subset \R\) denotes, as before, a bounded interval. 

\begin{definition}\label{def:weak_sol-Ellis}
Let \(\alpha>1\). For a given $T > 0$ a weak solution to \eqref{eq:PDE-Ellis} is defined as a function
\begin{equation*}
    u \in C\bigl([0,T];H^1(\Omega)\bigr) 
    \cap
    L_{\alpha+1}\bigl((0,T);W^3_{\alpha+1,B}(\Omega)\bigr)
    \quad
    \text{with}
    \quad
    u_t \in L_\frac{\alpha+1}{\alpha}\bigl((0,T);(W^1_{\alpha+1,B}(\Omega))'\bigr)
\end{equation*} 
that has the following properties:
\begin{itemize}
    \item[(i)] (Weak formulation) $u$ satisfies the differential equation  $\eqref{eq:PDE-Ellis}_1$ in the weak sense, i.e., 
    \begin{equation*}
        \int_0^T \langle u_t,\phi\rangle_{W^1_{\alpha+1}(\Omega)}\, dt
        =
        \int_0^T \int_\Omega u^3\bigl(1+|uu_{xxx}|^{\alpha-1}\bigr)u_{xxx}\, \phi_x\, dx\, dt
    \end{equation*}
    for all test functions $\phi \in L_{\alpha+1}\bigl((0,T);W^1_{\alpha+1,B}(\Omega)\bigr)$.
    \item[(ii)] (Initial and boundary values) $u$ satisfies the contact angle condition \(u_x=0\) on \(\partial\Omega\) and the initial condition $\eqref{eq:PDE-Ellis}_3$ pointwise.
\end{itemize}
\end{definition}

In the case of Ellis fluids we naturally obtain the dissipation functional
\begin{equation*}
    D[u] = \int_{\Omega} u^3\bigl(1+|uu_{xxx}|^{\alpha-1}\bigr)|u_{xxx}|^2 \, dx.
\end{equation*}

For general positive initial data in \(H^1(\Omega)\) we can show existence of local-in-time positive weak solutions.

\begin{theorem}[Local existence of positive weak solutions] \label{thm:Local_Ex_Ellis}
Let \(\alpha > 1\). Given a positive initial value $u_0\in H^1(\Omega)$ with
$u_0(x) > 0,\ x \in \bar{\Omega}$,
% \begin{equation*}
%     \frac{1}{|\Omega|} \int_\Omega u_0(x)\, dx = \bar{u}_0
%     \quad \text{and} \quad
%     \|u_0 - \bar{u}_0\|_{H^1(\Omega)} \leq \eps
% \end{equation*}
there exists a time $T > 0$ such that problem \eqref{eq:PDE-Ellis} admits at least one positive weak solution \begin{equation*}
    u \in C\bigl([0,T];H^1(\Omega)\bigr) 
    \cap
    L_{\alpha+1}\bigl((0,T);W^3_{\alpha+1,B}(\Omega)\bigr)
    \quad \text{with} \quad
    u_t \in 
    L_\frac{\alpha+1}{\alpha}\bigl((0,T);(W^1_{\alpha+1,B}(\Omega))^\prime\bigr)
\end{equation*}
on $(0,T)$ in the sense of Definition \ref{def:weak_sol-Ellis}. Moreover, such a solution has the following properties:
\begin{itemize}
    \item[(i)] (Positivity) $u$ is bounded away from zero
    \begin{equation*}
        0
        < C_T \leq
        u(t,x),
        \quad
        0 \leq t \leq T,\ x \in \bar{\Omega}.
    \end{equation*}
    \item[(ii)] (Conservation of mass) $u$ conserves its mass in the sense that
        \begin{equation*}
            \|u(t)\|_{L_1(\Omega)} 
            =
            \|u_0\|_{L_1(\Omega)},
            \quad 0 \leq t \leq T.
        \end{equation*}
    \item[(iii)] (Energy-dissipation identity) Energy is dissipated along solutions
        \begin{equation}\label{eq:energy-dissipation-Ellis}
            E[u](t) + \int_0^t D[u](s) \, ds= E[u_0]
        \end{equation}
    for almost every $t\in [0,T]$. 
\end{itemize}
\end{theorem}

\begin{remark}
    For positive initial datum $u_0 \in W^{4\rho}_{\alpha+1,B}(\Omega),\ 4\rho > 3+1/(\alpha+1)$ with $u_0(x) >0$, $x\in \bar{\Omega}$, the problem \eqref{eq:PDE-Ellis} actually possesses a unique maximal strong solution \cite{lienstromberg_local_2020}
\begin{equation*}
	u \in C\bigl([0,T_{\text{max}});W^{4\rho}_{\alpha+1,B}(\Omega)\bigr)
	\cap
	C^{\rho}\bigl([0,T_{\text{max}});L_{\alpha+1}(\Omega)\bigr)
	\cap
	C\bigl((0,T_{\text{max}});W^4_{\alpha+1,B}(\Omega)\bigr)
	\cap
	C^1\bigl((0,T_{\text{max}});L_{\alpha+1}(\Omega)\bigr).
\end{equation*}
Moreover, the solution enjoys the following properties:
\begin{itemize}
    \item[(i)] (Positivity) $u$ is positive
        \begin{equation*}
        	u(t,x) > 0,
            \quad
            0 \leq t < T_{\text{max}},\, x\in \bar{\Omega}.
            \end{equation*}
    \item[(ii)] (Conservation of mass) $u$ conserves its mass in the sense that
        \begin{equation}\label{eq:cons-mass-sigma-Ellis-strong}
            \|u(t)\|_{L_1(\Omega)} 
            =
            \|u_0\|_{L_1(\Omega)},
            \quad 0 \leq t < T_{\text{max}}.
        \end{equation}
    \item[(iii)] (Energy-dissipation identity) $u$ satisfies the energy-dissipation identity
    \begin{equation}\label{eq:energy-diss-sigma-Ellis-strong}
        E[u](t) + \int_0^t D[u](s) 
        =
        E[u_0],
        \quad
        0 \leq t < T_{\text{max}}.
    \end{equation}
    \item[(iv)] (Maximal time of existence) Suppose that $T_{\text{max}} < \infty$. Then
        \begin{equation*}
            \liminf_{t \nearrow T_{\text{max}}} \frac{1}{\min_{x \in \bar{\Omega}} u(t)} 
            + 
            \|u(t)\|_{W^{4\gamma}_{\alpha+1,B}(\Omega)} 
            = 
            \infty
        \end{equation*}
    for all $\gamma \in (\theta,1]$.
\end{itemize}
\end{remark}

\begin{proof}[\textbf{Proof of Theorem \ref{thm:Local_Ex_Ellis}}]
For initial data $u_0 \in W^{4\rho}_{\alpha+1,B}(\Omega),\ 4\rho > 3+1/(\alpha+1)$ with $u_0(x) >0$ we obtain local-in-time strong solutions. Choosing a sequence $\big(u_0^{(k)}\big)_{k \in \mathbb{N}}$ with $u_0^{(k)}(x) > 0,\ x \in \bar{\Omega}$, and $\bar{u}_0^{(k)} = \bar{u}_0$ such that $u_0^{(k)} \to u_0$ strongly in $H^1(\Omega)$ guarantees, together with the energy-dissipation identity \eqref{eq:energy-diss-sigma-Ellis-strong} and similar a-priori bounds as in Lemma \ref{lem:uniform_bounds} that the corresponding strong solutions \(u^{(k)}\) converge weakly in \(L_{\infty}\bigl((0,T);H^1(\Omega)\bigr)\cap L_{\alpha+1}\bigr((0,T);W^3_{\alpha+1,B}(\Omega)\bigr)\) to a weak solution $u \in C\bigl([0,T];H^1(\Omega)\bigr) 
    \cap
    L_{\alpha+1}\bigl((0,T);W^3_{\alpha+1,B}(\Omega)\bigr)$. Positivity, conservation of mass and the energy-dissipation identity for almost every \(t\in [0,T]\) are preserved under taking the weak limit.
\end{proof}

%---------------------------------
\bigskip
    
\noindent\textbf{\textsc{Steady states of \eqref{eq:PDE-Ellis}. }} We now turn to stability. First, we find that the same characterisation of positive steady states as before holds true. This is the content of the following theorem which has already been proved in  \cite[Corollary 6.3]{lienstromberg_local_2020}.

\begin{theorem}[Characterisation of positive steady states]
A function $u \in W^3_{\alpha+1,B}(\Omega)$ is a positive steady-state solution of \eqref{eq:PDE-Ellis} if and only if $u \equiv u_\ast \in \R_{>0}$ is given by positive constant.
\end{theorem}

\noindent\textbf{\textsc{Global existence and exponential stability for \eqref{eq:PDE-Ellis}. }}
It is well-known that for the Newtonian thin-film equation
\begin{equation} \label{eq:PDE-Newtonian}
    \begin{cases}
        u_t + \bigl(u^{3} u_{xxx}\bigr)_x
        =
        0,
        & t>0,\ x \in \Omega,
        \\
        u_x(t,x) = u_{xxx}(t,x) = 0, 
        &
        t > 0,\ x \in \partial\Omega,
        \\
        u(0,x) = u_0(x),
        &
        x \in \Omega,
    \end{cases}
\end{equation}
solutions close to positive steady states converge exponentially fast to equilibrium \cite{bertozzi_lubrication_1996}. We now prove that the same behaviour can be found for Ellis-law thin films. 

\begin{theorem}[Global existence and exponential stability]\label{Thm:Shear_thin_Ellis}
    Fix \(1<\alpha < \infty\). There exists \(\eps>0\) such that, for all positive initial values \(u_0\in H^1(\Omega)\) with \(\|u_0-\bar{u}_0\|_{H^1(\Omega)} < \eps\), there is a global positive weak solution
    \begin{equation*}
        u\in C\bigl([0,\infty);H^1(\Omega)\bigr) \cap L_{\alpha+1,\mathrm{loc}}\bigl((0,\infty);W^3_{\alpha+1,B}(\Omega)\bigr) \quad  \text{with} \quad
    u_t \in L_{\frac{\alpha+1}{\alpha},\text{loc}}\bigl((0,\infty);(W^1_{\alpha+1,B}(\Omega))'\bigr).
    \end{equation*}
    Moreover, there is \(\lambda >0\) and a constant \(C>0\) such that
        \[\|u(t)-\bar{u}_0\|_{H^1(\Omega)} \leq C e^{-\lambda t} \|u_0\|_{H^1(\Omega)}.\]
    Furthermore, we find that the dissipation decreases exponentially along the solution in the following \(L_1\)-in-time sense:
    \begin{equation}\label{eq:L1-in-time-dissipation-thin_Ellis}
        \int_{t/2}^{t} D[u](s) \, ds \leq C e^{-2\lambda t} \|u_0\|_{H^1(\Omega)}^2.
    \end{equation}
\end{theorem}
\begin{proof}
    Let \(u_0\in H^1(\Omega)\) with \( \bar{u}_0/2< u_0(x) < 2\bar{u}_0\), \(x\in \bar{\Omega}\) and \(u \in C\bigl([0,T];H^1(\Omega)\bigr) \cap L_{\alpha+1}\bigl((0,T);W^3_{\alpha+1,B}(\Omega)\bigr) \) the solution to \eqref{eq:PDE-Ellis} provided by Theorem \ref{thm:Local_Ex_Ellis}.
    We also define
    \begin{equation*}
    \tau
    =
    \sup\{\tilde{T} > 0;\ \text{$\exists$ a weak solution $u$ to \eqref{eq:PDE-Ellis} on } [0,\tilde{T}]
    \text{ with } \tfrac{1}{2} \bar{u}_0 \leq u(t,x) \leq 2 \bar{u}_0\ \forall\ 0 \leq t \leq \tilde{T}\}.
\end{equation*}
    Then \(\tau \leq T\) because otherwise we can extend weak solutions beyond time \(\tau\).
    
    Next, write \(u(t,x) = \bar{u}_0 + v(t,x)\) for \( (t,x) \in [0,T]\times \Omega \), where due to conservation of mass \(\int_{\Omega} v(t,x) \, dx= 0\) for all \(t\in [0,T]\). Then, by continuity and the definition of \(\tau\), we have \(|v(t,x)| \leq \bar{u}_0/2\) for \(0\leq t \leq \tau\).
    
    We then find, by the energy-dissipation identity \eqref{eq:energy-dissipation-Ellis} and the definition of \( \tau\), that
    \begin{align*}
        \frac{d}{dt} E[u](t) &= -D[u](t) = -\int_{\Omega}u^3(t,x)\bigl(1+|u(t,x)u_{xxx}(t,x)|^{\alpha-1}\bigr)|u_{xxx}(t,x)|^2 \, dx \\
        &\leq - \int_{\Omega} u^3(t,x) |u_{xxx}(t,x)|^2 \, dx \leq -C\int_{\Omega} |u_{xxx}(t,x)|^2 \, dx \leq - C E[u](t)
    \end{align*}
    for almost every \(t\in [0,\tau]\), where in the last step we have applied Lemma \ref{Lem:Poincare}. This yields that \(E[u](t)\) is decreasing and so \(\tau = T\). Applying Gronwall's inequality, we deduce that
    \begin{equation*}
        E[u](t) \leq E[u_0]e^{-Ct}
    \end{equation*}
    for all \(t\in [0,\tau]\). 
    Now choose \(\eps >0\) small enough so that
    \begin{equation*}
        \|u_0-\bar{u}_0\|_{L^{\infty}(\Omega)} \leq CE[u_0]^{\frac{1}{2}} \leq \frac{\bar{u}_0}{2},
    \end{equation*}
    where in the first estimate we have used the embedding $H^1(\Omega) \hookrightarrow L_{\infty}(\Omega)$ and Poincaré's inequality.
   Using this and the fact that \( E[u](t) \) is decreasing, we get 
   \begin{equation*}
        \|u(t)-\bar{u}_0\|_{L^{\infty}(\Omega)} \leq CE[u](t)^{\frac{1}{2}} \leq CE[u_0]^{\frac{1}{2}} \leq \frac{\bar{u}_0}{2}
    \end{equation*}
    for all \(t\in [0,T]\). 
    We can then extend the solution beyond time \(T\) to a global-in-time weak solution \(u\in C\bigl([0,\infty);H^1(\Omega)\bigr) \cap L_{\alpha+1,\mathrm{loc}}\bigl((0,\infty);W^3_{\alpha+1,B}(\Omega)\bigr)\) to \eqref{eq:PDE-Ellis} that satisfies 
    \begin{equation*}
        E[u](t) \leq  E[u_0] e^{-Ct}, \quad 0 \leq  t <\infty.
    \end{equation*}
    By Poincaré's inequality, we then conclude that
    \begin{equation*}
        \|u(t)-u_0\|_{H^1(\Omega)} \leq C E[u(t)]^{1/2} \leq C\|\nabla u_0\|_{L^2(\Omega)} e^{-\lambda t},
    \end{equation*}
    for some \(\lambda >0\) and all \(t\in (0,\infty)\).
    
    The $L_1$-in-time estimate follows from adapting Theorem \ref{thm:L1-in-time} to the new dissipation functional.
\end{proof}

%=============================================================================
%=============================================================================
%=============================================================================

%-------------------------------------------------------
%-------------------------------------------------------
%-------------------------------------------------------

\appendix
\addcontentsline{toc}{section}{APPENDICES}

%-------------------------------------------------------
%-------------------------------------------------------

\section{Proof of Lemma \ref{lem:convergence} and Lemma \ref{lem:limit_flux}}

Here we give precise proofs of the auxiliary results needed to establish local existence of positive weak solutions to the original problem \eqref{eq:PDE} in Section \ref{sec:Local_Existence}. 

%-------------------------------------------------------
%-------------------------------------------------------

%-------------------------------------------------------
%-------------------------------------------------------

\begin{proof}[\textbf{Proof of Lemma \ref{lem:convergence}}]
 \noindent\textbf{(i)}  In Lemma \ref{lem:uniform_bounds} (i) and (iii) we have shown that 
\begin{equation*}
	\begin{cases}
		(u^\sigma)_\sigma \text{ is uniformly bounded in } L_\infty\bigl((0,T);H^1(\Omega)\bigr) %\cap L_{\alpha+1}\bigl((0,T);W^3_{\alpha+1}(S^1)\bigr),
		& \\
		(u^\sigma_t)_\sigma \text{ is uniformly bounded in } L_\frac{\alpha+1}{\alpha}\bigl((0,T);(W^1_{\alpha+1,B}(\Omega))'\bigr).&
    \end{cases}
\end{equation*}
Moreover, in view of the Rellich-Kondrachov theorem, see e.g. \cite[Thm. 6.3]{adams_sobolev_2003}, we have
\begin{equation*}
	H^1(\Omega) \xhookrightarrow[]{c} C^{\rho}(\bar{\Omega}) \hookrightarrow (W^1_{\alpha+1}(\Omega))', \quad \rho \in [0,1/2),
\end{equation*}
where $\xhookrightarrow[]{c}$ indicates compactness of the embedding.
This enables us to use \cite[Cor. 4]{simon_compact_1986}, which gives that the sequence 
\begin{equation*}
	(u^\sigma)_\sigma \text{ is relatively compact in } C\bigl([0,T];C^\rho(\bar{\Omega})\bigr) 
\end{equation*}
with $\rho \in [0,1/2)$ as above.
	    
 \noindent\textbf{(ii)}  This is an immediate consequence of Lemma \ref{lem:uniform_bounds} (ii).
	
 \noindent\textbf{(iii)}  By Lemma \ref{lem:uniform_bounds} (iii), we can extract a subsequence $(u^\sigma_t)_\sigma$ such that
\begin{equation*}
	u^\sigma_t \rightharpoonup v 
	\quad \text{weakly in } L_\frac{\alpha+1}{\alpha}\bigl((0,T);(W^1_{\alpha+1,B}(\Omega))'\bigr) \hookrightarrow \Dcal'\bigl((0,T);(W^1_{\alpha+1,B}(\Omega))'\bigr)
\end{equation*}
for some limit function $v \in L_\frac{\alpha+1}{\alpha}\bigl((0,T);(W^1_{\alpha+1,B}(\Omega))'\bigr)$.
Since, in addition,  
\begin{equation*}
    u^\sigma \longrightarrow u \quad \text{in }  C\bigl([0,T];C^\rho(\bar{\Omega})\bigr) \hookrightarrow \Dcal'\bigl((0,T);(W^1_{\alpha+1,B}(\Omega))'\bigr) \quad 
\rho \in [0,1/2),    
\end{equation*}
we conclude that 
\begin{equation*}
	u^\sigma_t \longrightarrow 
	u_t
	\quad \text{in }
	\Dcal'\bigl((0,T);(W^1_{\alpha+1,B}(\Omega))'\bigr)\bigr),
\end{equation*}
and thus, $v = u_t \in L_\frac{\alpha+1}{\alpha}\bigl((0,T);(W^1_{\alpha+1,B}(\Omega))'\bigr)$.
	
 \noindent\textbf{(iv)}  Note that the strong convergence $u^\sigma \to u$ in $C\bigl([0,T];C^\rho(\bar{\Omega})\bigr),\, \rho \in [0,1/2),$ in (i) implies uniform convergence 
\begin{equation}\label{eq:conv_1}
    u^\sigma 
    \longrightarrow 
    u 
    \quad
    \text{in } 
    C\bigl([0,T]\times \bar{\Omega}\bigr).
\end{equation}
Moreover, by Lemma \ref{lem:uniform_bounds} (v), there exists  some $\hat{u} \in L_{\alpha+1}\bigl((0,T);W^3_{\alpha+1,B}(\Omega)\bigr)$ such that
\begin{equation} \label{eq:conv_2}
    u^\sigma \rightharpoonup \hat{u}
	\quad
	\text{in } L_{\alpha+1}\bigl((0,T);W^3_{\alpha+1,B}(\Omega)\bigr).
\end{equation}
Because of the uniqueness of the limit function, we infer from  \eqref{eq:conv_1} and \eqref{eq:conv_2} that
\begin{equation*}
     u^\sigma \rightharpoonup u
    \quad
    \text{in } 
    L_{\alpha+1}\bigl((0,T);W^3_{\alpha+1,B}(\Omega)\bigr).
\end{equation*}
In virtue of the weak lower-semicontinuity of the norm and Lemma \ref{lem:uniform_bounds} (iv) and (v), we finally obtain
\begin{equation} \label{eq:lsc}
	\begin{cases}
	    \|u_{xxx}\|_{L_{\alpha+1}((0,T)\times \Omega)}
	    \leq
	    \liminf_{\sigma \to 0} \|u^\sigma_{xxx}\|_{L_{\alpha+1}((0,T)\times \Omega)}
	    \leq 
	    C
	    & \\
	    \|u\|_{L_{\alpha+1}((0,T);W^3_{\alpha+1,B}(\Omega))}
	    \leq
	    \liminf_{\sigma \to 0} \| u^\sigma\|_{L_{\alpha+1}((0,T);W^3_{\alpha+1,B}(\Omega))}
	    \leq 
	    C &
	\end{cases}
\end{equation}
for some generic constant $C > 0$ that is independent of $\sigma$. 

 \noindent\textbf{(v)}  This follows by reasoning similarly to (iii) and the proof is complete.
\end{proof}

%-------------------------------------------------------
%-------------------------------------------------------

\begin{proof}[\textbf{Proof of Lemma \ref{lem:limit_flux}}]
The proof is divided into several steps. Throughout the proof, when there is no fear of ambiguity, we pass to a subsequence without relabelling it.
	
 \noindent\textbf{(i)}  First, by Lemma \ref{lem:convergence} (ii), we know that $|u^\sigma|^{\alpha+2} \psi_\sigma(u^\sigma_{xxx}\bigr)$ is weakly sequentially compact, i.e., there is an element $\chi \in L_\frac{\alpha+1}{\alpha}\bigl((0,T)\times \Omega)\bigr)$ such that
\begin{equation*}
    |u^\sigma|^{\alpha+2} \psi_\sigma(u^\sigma_{xxx}) 
    \rightharpoonup
    \chi
    \quad
	\text{weakly in } 
	L_\frac{\alpha+1}{\alpha}\bigl((0,T)\times \Omega)\bigr).
\end{equation*}
It remains to identify the limit flux $\chi$.
	
 \noindent\textbf{(ii)}  Next, in view of Lemma \ref{lem:uniform_bounds} (v) and the lower semicontinuity of the norm,
\begin{equation*}
    u_x \in L_{\alpha+1}\bigl((0,T);W^1_{\alpha+1,0}(\Omega) \cap W^2_{\alpha+1}(\Omega)\bigr) . 
\end{equation*}
Thus, we can take $\varphi = u_{xx} \in L_{\alpha+1}\bigl((0,T);W^1_{\alpha+1}(\Omega)\bigr)$ as a test function in the equation \eqref{eq:PDE_reg} for $u^\sigma$. This gives
\begin{equation*}
    \int_0^{T} \int_{\Omega} u^\sigma_t u_{xx}\, dx\, dt 
    + 
    \int_0^{T} \int_{\Omega} |u^\sigma|^{\alpha+2} \psi_\sigma(u^\sigma_{xxx}) 
    u_{xxx}\, dx\, dt 
    = 0.
\end{equation*}
Using Lemma \ref{lem:convergence} (iii), the first term satisfies
\begin{equation*}
    \int_0^{T} \int_{\Omega} u^\sigma_t\,  u_{xx}\, dx\, dt 
    \longrightarrow 
	\int_0^{T} \int_{\Omega} u_t\,  u_{xx}\, dx\, dt
	=
	E[u](T) - E[u](0)
\end{equation*}
as $\sigma \searrow 0$. For the second term, we infer from Lemma \ref{lem:convergence} (ii) that
\begin{equation*}
    \int_0^{T} \int_{\Omega} |u^\sigma|^{\alpha+2} \psi_\sigma(u^\sigma_{xxx}) u_{xxx}\, dx\, dt
	\longrightarrow 
	\int_0^{T} \int_{\Omega} \chi u_{xxx}\, dx\, dt,
\end{equation*}
as $\sigma \searrow 0$. Consequently, we obtain the identity
\begin{equation*}
	E[u](t) + \left\langle \chi | u_{xxx} \right\rangle_{L_{\alpha+1}}
	=
	E[u_0]
\end{equation*}
for almost every $t \in [0,T]$. 

 \noindent\textbf{(iii)}  We now use Minty's trick to identify the limit flux $\chi$. Note that the operator 
\begin{equation*}
	\begin{cases}
	    \psi_\sigma \colon L_{\alpha+1}\bigl((0,T)\times \Omega\bigr) 
	    \longrightarrow L_{\frac{\alpha+1}{\alpha}}\bigl((0,T)\times \Omega\bigr), & \\
    	\psi_\sigma(v) = \bigl(|v|^2 + \sigma^2\bigr)^\frac{\alpha-1}{2} v &
    \end{cases}
\end{equation*}
is monotone, i.e. for all $v,w \in L_{\alpha+1}\bigl((0,T)\times \Omega\bigr)$ with $v\neq w$ it holds that 
\begin{equation*}
	\langle \psi_\sigma(v) - \psi_\sigma(w)|v-w\rangle_{L_{\alpha+1}}
	=
	\int_0^T \int_{\Omega} \bigl(\psi_\sigma(v) - \psi_\sigma(w)\bigr) (v-w)\, dx\, dt
	> 0.
\end{equation*} 
This follows immediately from the monotonicity of the function 
$\psi_{\sigma}: \R \rightarrow \R$, $s \mapsto (s^2+ \sigma^2)^\frac{\alpha-1}{2} s$. 
From now on, we simply write $\langle v | w\rangle$ for the dual pairing $\langle v | w \rangle_{L_{\alpha+1}((0,T)\times \Omega)}$ between $v \in L_\frac{\alpha+1}{\alpha}\bigl((0,T)\times \Omega\bigr)$ and $w \in L_{\alpha+1}\bigl((0,T)\times \Omega\bigr)$. Let now $\phi \in W^3_{\alpha+1}\bigl((0,T)\times \Omega\bigr)$. In view of the monotonicity of $\psi_\sigma$, we have
\begin{align*}
	0
	&\leq
	\left\langle |u^\sigma|^{\alpha+2} \psi_\sigma(u^\sigma_{xxx}) - |u^\sigma|^{\alpha+2} \psi_\sigma(\phi_{xxx}) | (u^\sigma-\phi)_{xxx} 
	\right\rangle
	\\
	&=	
	\left\langle |u^\sigma|^{\alpha+2} \psi_\sigma(u^\sigma_{xxx}) | u^\sigma_{xxx} \right\rangle
	-
	\left\langle |u^\sigma|^{\alpha+2} \psi_\sigma(u^\sigma_{xxx}) | \phi_{xxx} \right\rangle
	\\
	&\quad
	-
	\left\langle |u^\sigma|^{\alpha+2} \psi_\sigma(\phi_{xxx}) | u^\sigma_{xxx} \right\rangle
	+
	\left\langle |u^\sigma|^{\alpha+2} \psi_\sigma(\phi_{xxx}) | \phi_{xxx} \right\rangle.
\end{align*}
We consider the four dual pairings on the right-hand side separately.

First, we rewrite the energy-dissipation identity for the problem \eqref{eq:PDE_reg} as 
\begin{equation*}
	\left\langle |u^\sigma|^{\alpha+2} \psi_\sigma(u^\sigma_{xxx}) | u^\sigma_{xxx}
	\right\rangle =
	E[u_0] - E[u^\sigma](t) \quad \text{for almost every } t\in [0,T]. 
\end{equation*}
Thanks to Lemma \ref{lem:convergence} (i) we know that $u^\sigma(t) \to u(t)$ in $H^1(\Omega)$ for almost every $t \in [0,T]$, and hence, as $\sigma \searrow 0$,  we have
\begin{equation} \label{eq:Minty_1}
	\left\langle |u^\sigma|^{\alpha+2} \psi_\sigma(u^\sigma_{xxx}) | u^\sigma_{xxx}
	\right\rangle
	\longrightarrow
	E[u_0] - E[u](t) \quad \text{for almost every } t\in [0,T].
\end{equation} 
For the second dual pairing, we get from Lemma \ref{lem:convergence} (ii) that 
\begin{equation}\label{eq:Minty_2}
	\left\langle |u^\sigma|^{\alpha+2} \psi_\sigma(u^\sigma_{xxx}) | \phi_{xxx} \right\rangle
	\longrightarrow
	\left\langle \chi | \phi_{xxx} \right\rangle,
	\quad
	\text{as } \sigma \searrow 0. 
\end{equation}
For the third pairing, we use Lemma \ref{lem:convergence} (i) and (iv) to obtain
\begin{equation*}
	\begin{cases}
		u^\sigma \longrightarrow u 
		\quad \text{strongly in } C\bigl([0,T]\times \bar{\Omega}\bigr) &
		\\
		u^\sigma_{xxx}
		\rightharpoonup
		u_{xxx}
		\quad \text{weakly in } L_{\alpha+1}\bigl((0,T)\times \Omega\bigr), &
	\end{cases}
\end{equation*}
and this implies
\begin{equation}\label{eq:Minty_3}
	\left\langle |u^\sigma|^{\alpha+2} \psi_\sigma(\phi_{xxx}) | u^\sigma_{xxx} \right\rangle
	\longrightarrow
	\left\langle |u|^{\alpha+2} \psi(\phi_{xxx}) | u_{xxx} \right\rangle,
    \quad
    \text{as } \sigma \searrow 0.
\end{equation}
Clearly, for the fourth pairing, we have
\begin{equation}\label{eq:Minty_4}
    \left\langle |u^\sigma|^{\alpha+2} \psi_\sigma(\phi_{xxx}) | \phi_{xxx} \right\rangle  
    \longrightarrow
    \left\langle |u|^{\alpha+2} \psi(\phi_{xxx}) | \phi_{xxx} \right\rangle,
    \quad
    \text{as } \sigma \searrow 0.
\end{equation} 
Combining \eqref{eq:Minty_1}--\eqref{eq:Minty_4} yields the inequality
\begin{equation*}
	0 
	\leq
	E[u_0] - E[u](t)
	-
	\left\langle \chi | \phi_{xxx} \right\rangle
	-
	\left\langle |u|^{\alpha+2} \psi(\phi_{xxx}) | (u - \phi)_{xxx}  \right\rangle,
\end{equation*}
and taking into account the identity 
\begin{equation*}
	E[u](t) + \left\langle \chi | u_{xxx} \right\rangle
	=
	E[u_0]
\end{equation*}
proved in step (ii), for almost every $t \in [0,T]$, we 
get that
\begin{equation*}
	0 
	\leq
	\left\langle \chi - |u|^{\alpha+2} \psi(\phi_{xxx}) | (u - \phi)_{xxx} \right\rangle.
\end{equation*}
Choosing $\phi = u - \lambda v$ for some arbitrary $v \in W^3_{\alpha+1}\bigl((0,T)\times \Omega\bigr)$ and $\lambda > 0$, gives the inequality
\begin{equation*}
	\left\langle \chi - |u|^{\alpha+2} \psi\bigl((u - \lambda v)_{xxx}\bigr) | v_{xxx} \right\rangle
	\geq 0
\end{equation*}
and thus in the limit $\lambda \searrow 0$ we deduce
\begin{equation*}
	\left\langle \chi - |u|^{\alpha+2} \psi(u_{xxx}) | v_{xxx} \right\rangle
	\geq 0,
	\quad
	v \in W^3_{\alpha+1}\bigl((0,T)\times \Omega\bigr),
\end{equation*}
for almost every $t \in [0,T]$. Now taking $\phi = u + \lambda v$, we see that
\begin{equation*}
	\left\langle \chi - |u|^{\alpha+2} \psi(u_{xxx}) | v_{xxx} \right\rangle
	\leq 0,
	\quad
	v \in W^3_{\alpha+1}\bigl((0,T)\times \Omega\bigr). 
\end{equation*}
Hence, we have shown that
\begin{equation*}
	\left\langle \chi - |u|^{\alpha+2} \psi(u_{xxx}) | v_{xxx} \right\rangle
	= 0,
	\quad
	v \in W^3_{\alpha+1}\bigl((0,T)\times \Omega\bigr),
\end{equation*}
from which, since $v \in W^3_{\alpha+1}\bigl((0,T)\times \Omega\bigr)$ is arbitrary, we are able to identify
\begin{equation*}
	\chi 
	=
	|u|^{\alpha+2} \psi(u_{xxx})
	\in
	L_\frac{\alpha+1}{\alpha}\bigl((0,T)\times \Omega\bigr).
\end{equation*}
This completes the proof.
\end{proof}

%=============================================================================
%=============================================================================
\section*{Acknowledgements}
\noindent KN is partially supported by the Austrian Science Fund (FWF) project  F\,65. JJ and CL have been partially supported by the Deutsche Forschungsgemeinschaft (DFG, German Research Foundation) through the collaborative research centre `The mathematics of emerging effects' (CRC 1060, Project-ID  211504053) and the Hausdorff Center for Mathematics (GZ 2047/1, Project-ID 390685813).

\printbibliography

\end{document}